\documentclass[a4paper,11pt]{amsart}
\usepackage{caption}
\usepackage{subcaption}
\usepackage[english]{babel,varioref}
\usepackage{epsfig}
\usepackage{graphicx}
\usepackage{a4wide}
\usepackage{url}
\usepackage{srcltx}
\usepackage{color}
\usepackage{hyperref}
\newtheorem{df}{Definition}[section]

\newtheorem{thm}[df]{Theorem}
\newtheorem*{ass}{Assumption}
\newtheorem{lem}[df]{Lemma}
\newtheorem{prop}[df]{Proposition}

\newtheorem{rmrk}[df]{Remark} 
\begin{document}
\title[Additive Average Schwarz Method for CRFVE Discretization]{Additive average Schwarz method for a Crouzeix-Raviart Finite Volume Element Discretization of Elliptic Problems with Heterogeneous Coefficients}
\author{Atle Loneland}\address{Department of Informatics, University of Bergen, 5020 Bergen, Norway}
\email{Atle.Loneland@ii.uib.no}
\author{Leszek Marcinkowski}\address{Faculty of Mathematics, University of Warsaw, Banacha 2, 02-097 Warszawa, Poland.} \thanks{L. Marcinkowski was partially supported by the Polish Scientific Grant 2011/01/B/ST1/01179.}
\email{Leszek.Marcinkowski@mimuw.edu.pl}
\author{Talal Rahman}\address{Department of Computing, Mathematics and Physics, Bergen University College, 5020 Bergen, Norway}
\email{Talal.Rahman@hib.no}
% \institute {sfasdf}
%opening
\maketitle

\begin{abstract}
% We introduce an additive Schwarz method for a Finite Volume Element (FVE) discretization of elliptic problem with Crouzeix-Raviart elements.
In this paper we introduce an additive Schwarz method for a Crouzeix-Raviart Finite Volume Element (CRFVE) discretization of a second order elliptic problem with discontinuous coefficients, where the discontinuities are both inside the subdomains and across and along the subdomain boundaries. We show that, depending on the distribution of the coefficient in the model problem, the parameters describing the GMRES convergence rate of the proposed method depend linearly or quadratically on the mesh parameters $H/h$.  Also, under certain restrictions on the distribution of the coefficient, the convergence of the GMRES method is independent of jumps in the coefficient.
\end{abstract}

\section{Introduction}
In this paper we introduce an additive Schwarz method for a second order elliptic problem with discontinuous coefficients inside the subdomains and across and along the subdomain boundaries. Problems of this type play a crucial part in the field of scientific computation For example, simulation of fluid flow in porous media are often affected by discontinuities in the permeability of the porous media. Discontinuities or jumps in the coefficient cause the performance of standard iterative methods to deteriorate as the discontinuities or the jumps increase.

The finite volume (FV) method is one of the most versatile discretization techniques used in computational fluid dynamics. It is widely used for the approximation of conservation laws, nonlinear problems and in convection-diffusion problems. The finite volume divides the domain into control volumes where the nodes from the finite difference or the finite element discretization are located in the control volume. Unlike the finite difference and the finite element method, the solution to the finite volume method satisfies conservation of certain quantities such as mass, momentum, energy and species. This property is exactly satisfied for every control volume in the domain and also for the whole computational domain. An attractive feature of this method is that it is directly connected to the physics of the system. There are two types of finite volume methods: One which is based on the finite difference discretization, called the finite volume method and one which is based on the finite element discretization named the finite volume element (FVE) method. In the later the approximation of the solution is sought in a finite element space and can therefore be considered as a Petrov-Galerkin finite element method.

Due to the popularity of the finite volume element method in science and engineering, many results on the analysis of the FVE method have been published, cf. \cite{LIN:2013:FVEM,ewing2002accuracy,rui2008convergence,chatzipantelidis1999finite,Chatzipantelidis:2002:FVE} and many more. In \cite{BANK:1987:BOX}, the authors proved that for the Poisson equation on a polygonal domain in two dimension, the stiffness matrix of the FVE method is equal to the stiffness matrix of the FE method for very general grids. In \cite{HACKBUSCH:1989:FIRST}, the authors proved that for the general elliptic case for polygonal domains in two dimensions, the error between the FE solution and the FVE solution is of first order in the general case and
of second order for some special FVE schemes. Thus, some superconvergence results valid for the finite element method is also valid for the finite volume element method, cf. \cite{CAI:1991:ON,WU:2003:ERROR}. Finite volume element methods based on the lowest order nonconforming Crouzeix-Raviart elements have been studied in \cite{chatzipantelidis1999finite}, where the author proves optimal order error estimates in the $L^2$-norm and a mesh dependent $H^1$-norm for the FVE solution of elliptic problems. Later, the authors in \cite{ewing2002accuracy} showed that the accuracy of the FVE method for linear conforming elements can be affected by the regularities of the exact solution and the source term. They also developed an error estimation framework for the FVE method which treats the FVE method as a perturbation of the Galerkin finite element method. For an overview over recent developments of FVE methods, cf. \cite{LIN:2013:FVEM} and references therein. 

Additive Schwarz Methods (ASM) for solving elliptic problems discretized by the finite element method have been studied thoroughly, cf. \cite{smith1996domain,toselli2005domain}, but ASMs for conforming finite volume element (FVE) discretization have only been consider in \cite{chou2003domain,Zhang:2006:ODD}. For the CR finite element discretization, there exist several results for second order elliptic problems; cf. \cite{sarkis1997nonstandard,rahman2005additive,Brenner:1996:TLS,marcinkowski1999mortar}, but for the CRFVE discretization, no ASMs have been studied.

In recent years, many results regarding ASMs for problems with discontinuities coefficient, both across and along subdomain boundaries, have been studied. In \cite{Graham:2007:MULTDD}, the authors proposed a two level additive Schwarz method where the coarse space is based on the multiscale finite element functions introduced in \cite{Hou:1997:MultFEM}. Later, several authors have proposed two level additive Schwarz methods with coarse spaces based on spectral basis functions constructed from solving different types of generalized eigenvalue problems, cf. \cite{Galvis:2010:DDMULT1,Galvis:2010:DDMULT2,Dolean:2012:DirNeu,Spillane:2014:GENEO} and many more. These method are all overlapping methods based on the conforming finite element discretization with exotic coarse spaces where the coarse basis functions are discrete harmonic functions or spectral basis functions.

The ASM we consider in this paper differs from methods mentioned above in the sense that it is a non-overlapping method and the discretization is done using nonconforming finite volume elements. Also, the average coarse space employed in our method is based on approximate discrete harmonic functions. Therefore, the average coarse space is a computationally cheap approximation to the full discrete harmonic function space. Also, the method does not require a coarse grid triangulation, i.e. we are free to use arbitrary irregular subdomains.

The variant of the additive Schwarz method we consider in this paper was first introduced for conforming P1 elements in \cite{BJORSTAD:1996:AVG} and later formulated for a mortar method with the  Crouzeix-Raviart elements in \cite{rahman2005additive}. In \cite{dryja2010additive} the authors analyzed the method for a discontinuous Galerkin discretization. In this paper we consider the same additive Schwarz method for the Crouzeix-Raviart FVE method introduced in \cite{chatzipantelidis1999finite} and show that the method depends linearly or quadratically on the mesh parameters $H/h$, i.e., depending on the distribution of the coefficient in the model problem, the parameters describing the convergence of the GMRES method used to solve the preconditioned system depends linearly or quadratically on the mesh parameters. Under certain restrictions on the distribution of the coefficient, the convergence of the GMRES method is independent of jumps in the coefficient. Also, using the framework developed in \cite{ewing2002accuracy}, we prove the $H^1$ error estimates using the same techniques as in \cite{ewing2002accuracy,rui2008convergence}. This estimate is of optimal order if the exact solution of the elliptic problem under consideration is of  the  $H^2$ regularity. Last, we show both theoretically and numerically that, in general, for varying coefficients the finite volume element bilinear form, and hence the resulting finite volume element stiffness matrix, is non-symmetric.

The rest of this paper is organized as follows. In Section~\ref{sect:prelim} we define the differential problem and the discrete problem, both for the nonconforming finite element and the nonconforming finite volume element discretization. In Section~\ref{sect:gmres} we introduce the GMRES method for the preconditioned system and the corresponding parameters describing the convergence rate. In Section~\ref{sect:asm} we introduce the additive Schwarz methods and give a detailed convergence analysis of the GMRES convergence rate. In Section~\ref{sect:numres} we show some numerical results which confirms the theory developed in the previous sections.
\section{Prelimenaries}
\label{sect:prelim}
\subsection{The Model Problem}
We consider the following elliptic boundary value problem
% bounded simply connected $\Omega\in\mathcal{R}$
\begin{eqnarray}
\label{eq:modelproblem}
-\nabla\cdot(\alpha(x)\nabla u)&=&f \hspace{15 mm} \mathrm{ in }\; \Omega,\\ \nonumber
u&=&0 \hspace{15 mm} \mathrm{on}\; \partial\Omega.
\end{eqnarray}
Where $\Omega$ is a bounded convex domain in $\mathbb{R}^2$ and $f\in L^2(\Omega)$. 

The corresponding standard variational (weak) formulation is: Find $u \in H^1_0(\Omega)$  such that
\begin{equation}
\label{eq:weakformulation}
  a(u,v)=\int_\Omega f v \:dx \quad \forall  v \in H^1_0(\Omega),
\end{equation}
where 
$$a(u,v)= \int_{\Omega} \alpha(x)\nabla u\cdot\nabla v \: dx.$$

The coefficient $\alpha(x)$ has the property $\alpha\in W^{1,\infty}(D_j)$ with respect to a nonoverlapping partitioning of $\Omega$ into open, connected Lipschitz polytopes $\mathcal{D}:=\{D_j:j=1,\ldots,n\}$, that is, 
\begin{equation*}
 \bar\Omega=\bigcup_{j=1}^n\bar D_j.
\end{equation*}
We require that $|\alpha|_{1,\infty,D_j}\leq C$ for $j=1,\ldots,n$ and that $\alpha\geq \alpha_0$ for some positive constant $\alpha_0$. For simplicity of presentation we also require that $\alpha_0\geq 1$. This last property can always be achieved by scaling  of (\ref{eq:modelproblem}).

\subsection{Basic notation}
Throughout this paper we will use standard notations for the Sobolev spaces. We denote the space of functions that have weak derivatives of order $s$ in the space $L^{2}(\Omega)$, as $H^s(\Omega)$. The norm on the space $H^s(\Omega)$ is defined by $$\|u\|_{s,\Omega}=\|u\|_{s}=\left(\int_\Omega\sum_{|\alpha|\leq s}|D^\alpha u|^2\,dx\right)^{1/2}.$$ The space of functions with bounded weak derivatives of order $s$ is denoted by $W^{s,\infty}(\Omega)$ with the corresponding norm defined as 
$$\|u\|_{s,\infty,\Omega}=\|u\|_{s,\infty}=\max_{0\leq |\alpha|\leq s}\|D^\alpha u\|_{2}.$$ The subspace of $H^1(\Omega)$ with functions vanishing on the boundary $\partial\Omega$ in the sense of traces, is denoted by $H^1_0(\Omega)$. The duality pairing between $H^{-1}(\Omega)$ and $H^1_0(\Omega)$, denote by $(f,u)$ is the action of a functional $f\in H^{-1}(\Omega)$ on a function $u\in H^1_0(\Omega)$.

Consider a triangulation $\mathcal{T}_h$ of $\Omega$, consisting of closed triangle elements $K$ such that $\bar{\Omega}=\bigcup_{K\in \mathcal{T}_h}K$. Let $h_K$ be the diameter of $K$ and let $h=\max_{K\in\mathcal{T}_h} h_K$ be the largest diameter of the triangles $K\in\mathcal{T}_h$. 

We assume that the triangulation is defined in such way that $\partial K$'s are aligned with $\partial D_j$'s. This implies that the coefficient $\alpha(x)$ has the property that $\alpha\in W^{1,\infty}(K)$ for all $K\in\mathcal{T}_h$. In addition, we also require the triangulation $\mathcal{T}_h(\Omega)$ to be quasiuniform \cite{brenner2008mathematical}.

We define the broken $H^1(\Omega)$-norm and $H^1(\Omega)$-seminorm respectively as
\begin{eqnarray*}
 \|v\|_{s,h,\Omega}=\left(\sum_{K\in\mathcal{T}_h}\|v\|^2_{s,K}\right)^{1/2}\text{ and }\quad |v|_{s,h,\Omega}=\left(\sum_{K\in\mathcal{T}_h}|v|^2_{s,K}\right)^{1/2}.
\end{eqnarray*}
We also introduce the energy seminorm $$\|u\|_{a,G}^2=\int_G \alpha(x)|\nabla u|^2\:dx$$
for any $G\subset \Omega$ and let $\|u\|_a=\|u\|_{a,\Omega}$. 

Let $\mathrm{E}_h(K)$ be the set of edges of $K\in\mathcal{T}_h$ and $\mathrm{E}_h=\cup_{K\in\mathcal{T}_h}\mathrm{E}_h(K)$, i.e. the union of all edges in the triangulation $\mathcal{T}_h$. Also, define $\mathrm{E}_h^\mathrm{in}$ as the set of interior edges of the triangulation $\mathcal{T}_h$, i.e. $e\in \mathrm{E}_h^\mathrm{in}$ if and only if $e\in \mathrm{E}_h$ and $e\not\subset \partial\Omega$. For every edge $e\in \mathrm{E}_h^\mathrm{in}$ we identify a region $V_e$ as the union of the two triangles ${K^{+e}\text{ and }K^{-e}\in\mathcal{T}_h}$ sharing $e$ as their common edge. Associated with this region, let $\mathcal{T}_h(V_e)$ be the set of the triangles of $V_e$ and $m_e$ the middle point of the edge $e\in E_h$ (cf. Figure \ref{fig:cv}).

Based on this triangulation $\mathcal{T}_h$, we introduce a dual mesh $\mathcal{T}_h^*$ consisting of elements called the control volumes. There are several ways to construct the dual mesh. We choose here to construct the dual mesh in the following way. Let $z_k$ be an interior point of $K\in\mathcal{T}_h$, we connect it with straight lines to the vertices of $K$ such that $K$ is partitioned into three subtriangles, $K_e$ for each edge $e\in\mathrm{E}_h(K)$. Denote this new finer triangulation of $\Omega$ by $\widetilde{\mathcal{T}_h}$ and let, for every $K\in\mathcal{T}_h$, $\widetilde{\mathcal{T}_h}(K)=\{\tilde{K}\in\widetilde{\mathcal{T}_h}:\tilde{K}\text{ subtriangle of }K\}$ be the set of subtriangles of $K$. 

We now associate with each edge $e\in \mathrm{E}_h^\mathrm{in}$ a corresponding control volume $b_e$ consisting of the two subtriangles of $\widetilde{\mathcal{T}_h}$ which have $e$ as an common edge. Define  $\mathcal{B}_e=\{b_e:e\in\mathrm{E}_h^\mathrm{in}\}$ to be the set of all 
such control volumes, and let $n_e$ be the normal vector corresponding to the edge $e$ in $K^{+e}$ of the two triangles $K^{+e}$ and $K^{-e}$ sharing $e$. 

We assume that there exists another  nonoverlapping partitioning of $\Omega$ into open, connected Lipschitz polytopes $\Omega_i$ such that 
$
   \overline{\Omega}=\bigcup_{i=1}^N \overline{\Omega}_i\,.
$
We also assume that these subdomains form a coarse triangulation of the domain which is shape regular as in \cite{Brenner:1999:BDD} and that the boundaries of elements in 
$T_h$ are aligned with the boundaries of any $\Omega_j$.

For notational convenience, we denote the CR nodal points, i.e. the midpoints of edges $e\in\mathrm{E}_h$, belonging to $\Omega, \Omega_i,\partial\Omega\text{ and } \partial\Omega_i$ by $\Omega^\mathrm{CR}_{h}, \Omega^\mathrm{CR}_{ih},\partial\Omega^\mathrm{CR}_{h}\text{ and } \partial\Omega^\mathrm{CR}_{ih}$, respectively. Correspondingly, the set of P1 conforming nodal points, i.e., vertices of elements in $\mathcal{T}_h(\Omega)$ are denoted by $\Omega_{h}, \Omega_{ih},\partial\Omega_{h}\text{ and } \partial\Omega_{ih}$, respectively. To simplify the presentation, we let $C$ be a generic positive constant independent of the mesh sizes $h$ and $H$, and of the functions under consideration. $C$ may be different at different occurrences.
\begin{figure}
% \caption{Control volume from \cite{chatzipantelidis1999finite}.}

\centering
\centering
\includegraphics[width=0.4\textwidth]{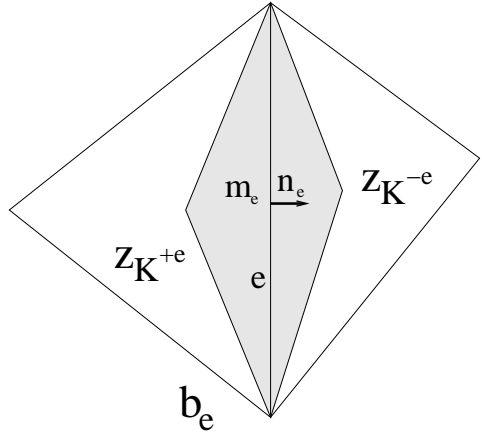}
 \caption{The control volume $b_e$ for an edge $e$ which is the common edge to the triangles $K^{+e}$ and $K^{-e}$. Here $m_e$ is the midpoint of $e$, $n_e$ normal unit vector to  $e$, $z_{K^{+e}}$ and $z_{K^{-e}}$ are the interior points of the  the triangles $K^{+e}$ and $K^{-e}$ which share the edge $e$.}
\label{fig:cv}
% \setlength{\unitlength}{1cm}
% \begin{picture}(1,1)
% \put(0,0){\line(1,2){2}}
% \put(0,0){\line(5,1){4}}
% 
% \end{picture}
% \end{figure}
% % Test example 1, see fig. \ref{inside}\\
% \begin{figure}[H]

\end{figure}
\subsection{The CRFVE method} Let $V_h$ be the nonconforming CR finite element space defined on the triangulation $\mathcal{T}_h$, 
$$ V_h=V_h(\Omega):=\{v\in L^2(\Omega): v_{|K}\in P_1, \quad  K \in T_h \quad v(m)=0\quad  m \in \partial\Omega_h^{CR}\},$$
and let $V_h^*$ be its dual control volume space
$$V_h^*=V_h^*(\Omega):=\{v\in L^2(\Omega): v_{|b_e}\in P_0,\quad  b_e \in \mathcal{T}_h^* \quad v(m)=0\quad m \in \partial\Omega_h^{CR}\} .$$
Obviously, $V_h=\text{span}\{\phi_e(x):e\in\mathrm{E}_h\}$ and $V_h^*=\text{span}\{\chi_e(x):e\in\mathrm{E}_h\}$, where $\{\phi_e\}$ are the standard nonconforming nodal basis functions and $\{\chi_e\}$ are the characteristic functions of the control volume $\{b_e\}$. Now, we introduce two interpolation operators, $I_h$ and $I_h^*$, defined for any function with properly defined and unique values at each midpoint $m \in  \Omega_h^{CR}$:, i.e.
$$I_hu=\sum_{e\in \mathrm{E}_h^{in}}u(m_e)\phi_e\quad \text{ and }\quad I_h^*u=\sum_{e\in \mathrm{E}_h^{in}}u(m_e)\chi_e.$$

We may then define the CRFVE approximation $u_h^{FV}$ of (\ref{eq:modelproblem}) as the solution to the following problem: Find $u_h^{FV}\in V_h$ such that 
\begin{equation}
\label{eq:dcrprb1}
 a_h^{FV}(u_h^{FV},I_h^*v)=\left(f,I_h^*v\right),\qquad v\in V_h
\end{equation}
or equivalently
\begin{equation}
\label{eq:dcrprb12}
  a_h^{FV}(u_h^{FV},v)=\left(f,v\right),\qquad v\in V_h^*,
\end{equation}
where the bilinear form is defined as
\begin{eqnarray}
 a_h^{FV}(u,v)=-\sum_{e\in \mathrm{E}_h^{in}}v(m_e)\int_{\partial b_e}\alpha(s)\nabla u\cdot\mathbf{n}\;ds\qquad u\in V_h, v\in V_h^*.
\end{eqnarray}
% and 
% \begin{eqnarray}
%  a_h^{FV}(u,I_h^*v)=-\sum_{e\in \mathrm{E}_h^{in}}v(m_e)\int_{\partial b_e}\alpha(x)\nabla u\cdot\mathbf{n}\;ds\qquad u,v\in V_h
% \end{eqnarray}
% Where $b_e$ is the control volume corresponding to $m_e$ (See figure [\ref{fig:cv}]).
The corresponding nonconforming finite element problem is defined as: Find $u_h^{FE}\in V_h$ such that
\begin{equation}
\label{eq:dcrprb2}
 a_h^{FE}(u_h^{FE},v)=\left(f,v\right),\qquad v\in V_h,
\end{equation}
where the CRFE bilinear form $a(\cdot,\cdot)$ is 
\begin{eqnarray}
 a_h^{FE}(u,v)=\sum_{K\in T_h}\int_K\alpha(x)\nabla u\cdot\nabla v\;dx,\qquad u,v\in V_h.
\end{eqnarray}
From the last bilinear form above we define a corresponding energy norm induced by $a_h^{FE}(\cdot,\cdot)$ as $\|\cdot\|_a=\sqrt{a_h^{FE}(\cdot,\cdot)}$.

% Expressing $u_h^{FV}$ as
% $$u_h^{FV}=\sum_{e\in \mathrm{E}_h^{in}} u_h^{FV}(m_e)\phi_e$$
% in terms of the nodal basis functions in $V_h$
% then the matrix representation of (\ref{eq:modelproblem}) is 
% \begin{equation}
%  \label{eq:matrixrep}
%  A=u_h^{FV}=b,
% \end{equation}
% where $

Now we state a lemma which is needed to prove the relationship between the CRFVE- and CRFE-bilinear forms for piecewise constant coefficients $\alpha(x)$. 
\begin{lem} Let $\alpha$ be piecewise constant over each element, i.e., $\alpha_K=\alpha(x)|_{K}$ is constant for each $K\in\mathcal{T}_h(\Omega)$, $e\in E_h^{\text{in}}\cap E_h(K)$ and $v\in V_h$. Then 
\begin{equation}
\label{eq:piececonst}
 \int_{b_e}\alpha(s)\frac{\partial u}{\partial n}\;ds=-\int_e\left[\frac{\partial u}{\partial n_e}\right]_\alpha\;ds.
\end{equation}
where $\left[\frac{\partial u}{\partial n_e}\right]_\alpha=\alpha_{K^{+e}}\frac{\partial u}{\partial n_e}-\alpha_{K^{-e}}\frac{\partial u}{\partial n_e}$ and $n_e$ is the normal vector of $K$ to $e$.
 
\end{lem}
\begin{proof}
 Let $v\in V_h$, $K\in \mathcal{T}_h(\Omega)$, $e\in E_h^{\text{in}}\cap E_h(K)$ and $n_e$ external normal vector of $K$ to $e$. Then we have
 \begin{equation*}
  \int_{\partial b_e}\alpha(s)\frac{\partial v}{\partial n}\;ds=\int_{\partial (b_e\cap K^{+e})}\alpha_{K^{+e}}\frac{\partial v}{\partial n}\;ds+\int_{\partial (b_e\cap K^{-e})}\alpha_{K^{-e}}\frac{\partial v}{\partial n}\;ds-\int_e\left[\frac{\partial u}{\partial n_e}\right]_\alpha\;ds.
 \end{equation*}
Using Green's formula and the fact that $\Delta v=0$ over $b_e\cap K^{+e}$ and $b_e\cap K^{-e}$ for any $e\in E_h^{\text{in}}$ we have
$$\int_{\partial(b_e\cap K^{+e})}\frac{\partial v}{\partial n}\;ds =\int_{b_e\cap K^{+e}}\triangle v\;ds=0,$$
and analogously for $\partial(b_e\cap K^{-e})$. From this we obtain (\ref{eq:piececonst}).
\end{proof}

The next lemma is a classical result:
\begin{lem} 
\label{lem:nrmeq}
There exists a constant C independent of h such that
\begin{equation*}
C^{-1}|v|_{1,h}^2\leq\sum_{K\in\mathcal{T}_h(\Omega)}\sum_{e,l\in \mathrm{E}_h(K)}(v(m_e)-v(m_l))^2\leq C|v|_{1,h}^2,\qquad \forall v\in V_h.
\end{equation*}

\end{lem}
The next lemma shows that if $\alpha$ is piecewise constant over fine elements then the CRFVE bilinear form is equal to the CRFE bilinear form, and in particular it is symmetric.
\begin{lem}
\label{lem:fveeqfe}
 Let $u,v\in V_h$, and let $\alpha_K$ be piecewise constant over each element $K\in\mathcal{T}_h(\Omega)$, then 
 \begin{equation}
  a_h^{FE}(u,v)=a_h^{FV}(u,I_h^*v).
 \end{equation}
\end{lem}
\begin{proof}
 We express $v$ as a linear combination of the basis elements of $V_h$, i.e. $v=\sum_{e\in \mathrm{E}_h^{in}} v(m_e)\phi_e.$ We may then write 
 \begin{eqnarray}
 \label{eq:1}
  a_h^{FE}(u,v)&=&\sum_{K\in\mathcal{T}_h}\alpha_K\int_K\nabla u\cdot\nabla v\;dx\nonumber \\
  &=&\sum_{e\in \mathrm{E}_h^{in}}v(m_e)\sum_{K\in\mathcal{T}_h(V_e)}\alpha_K\int_K\nabla u\cdot\nabla\phi_e\;dx
 \end{eqnarray}
 For each $e\in \mathrm{E}_h^{in}$ and $u\in V_h$, we have
 \begin{eqnarray}
 \label{eq:2}
 \sum_{K\in\mathcal{T}_h(V_e)}\alpha_K\int_K\nabla u\cdot\nabla \phi_e\;dx&=&\sum_{K\in\mathcal{T}_h(V_e)}\alpha_K\int_{\partial K}\frac{\partial u}{\partial n}\phi_e\;ds\nonumber\\
 &=&\alpha_{K^{+e}}\int_{\partial K^{+e}}\frac{\partial u}{\partial n}\phi_e\;ds+\alpha_{K^{-e}}\int_{\partial K^{-e}}\frac{\partial u}{\partial n}\phi_e\;ds\nonumber\\
 &=&\alpha_{K^{+e}}\int_{\partial K^{+e}\setminus e}\frac{\partial u}{\partial n}\phi_e\;ds+\alpha_{K^{-e}}\int_{\partial K^{-e}\setminus e}\frac{\partial u}{\partial n}\phi_e\;ds\nonumber\\
 &&+\alpha_{K^{+e}}\int_{e}\frac{\partial u}{\partial n_e}\phi_e\;ds-\alpha_{K^{-e}}\int_{e}\frac{\partial u}{\partial n_e}\phi_e\;ds\nonumber
%  &=&\left[\frac{\partial u}{\partial n_e}\right]_\alpha\int_e\phi_e\;ds\nonumber\\
 \end{eqnarray} 
Using the fact that $\phi_e$ is a linear polynomial and $\frac{\partial u}{\partial n}$ is constant on every side of $K\in\mathcal{T}_h(V_e)$  we get
 \begin{equation}
 \sum_{K\in\mathcal{T}_h(V_e)}\alpha_K\int_K\nabla u\cdot\nabla \phi_e\;dx=\int_e\left[\frac{\partial u}{\partial n_e}\right]_\alpha\;ds, 
 \end{equation}
 Combining (\ref{eq:1}) and (\ref{eq:2}) we obtain 
 
 \begin{eqnarray}
  \label{eq:3}
  a_h^{FE}(u,v)&=&\sum_{e\in \mathrm{E}_h^{in}}v(m_e)\int_e\left[\frac{\partial u}{\partial n_e}\right]_\alpha\;ds\nonumber\\
  &=&-\sum_{e\in \mathrm{E}_h^{in}}v(m_e)\int_{b_e}\alpha(s)\frac{\partial u}{\partial n}\;ds=a_h^{FV}(u,I_h^*v).
 \end{eqnarray}
which completes the proof.
\end{proof}
 
For varying coefficients in general, the FVE bilinear form is non-symmetric. This is easily seen by looking at $a_h^{FV}(\phi_i,I_h^*\phi_j)$ and $a_h^{FV}(\phi_j,I_h^*\phi_i)$. We state this as a remark.

\begin{rmrk}
For varying coefficients in (\ref{eq:modelproblem}), i.e. for a coefficient $\alpha$ which are not piecewise constant over each element, the FVE bilinear form is non-symmetric and hence in general we have 
$$a_h^{FV}(\phi_i,I_h^*\phi_j)\neq a_h^{FV}(\phi_j,I_h^*\phi_i),$$
for two nodal basis functions $\phi_j,\phi_i\in V_h(\Omega)$.
\end{rmrk}
\begin{proof}
Let $i,j,l$ be the three indices for the edges of a triangle $K\in\mathcal{T}_h$, then we have for $a_h^{FV}(\phi_i,I_h^*\phi_j)$
\begin{eqnarray}
\label{eq:nomsymmetry1}
a_h^{FV}(\phi_i,I_h^*\phi_j)&=&-\int_{\partial b_j}\alpha(s)\nabla \phi_i\cdot\mathbf{n}\;ds=-\int_{\partial( b_j\cap K)\cap\partial b_j}\alpha(s)\nabla \phi_i\cdot\mathbf{n}\;ds\nonumber\\
&=&-\nabla \phi_i\cdot\mathbf{n}_{jl}\int_{\partial(b_j\cap b_l)}\alpha(s)\;ds-\nabla \phi_i\cdot\mathbf{n}_{ji}\int_{\partial(b_j\cap b_i)}\alpha(s)\;ds
\end{eqnarray}
similarly for $a_h^{FV}(\phi_j,\phi_i)$ we have
\begin{eqnarray}
\label{eq:nomsymmetry2}
a_h^{FV}(\phi_j,I_h^*\phi_i)&=&-\int_{\partial b_i}\alpha(s)\nabla \phi_j\cdot\mathbf{n}\;ds=-\int_{\partial( b_i\cap K)\cap\partial b_i}\alpha(s)\nabla \phi_j\cdot\mathbf{n}\;ds\nonumber\\
&=&-\nabla \phi_j\cdot\mathbf{n}_{il}\int_{\partial(b_i\cap b_l)}\alpha(s)\;ds-\nabla \phi_j\cdot\mathbf{n}_{ij}\int_{\partial(b_i\cap b_j)}\alpha(s)\;ds
\end{eqnarray}
where $\mathbf{n}_{ij}$, $\mathbf{n}_{ji}$ ,$\mathbf{n}_{jl}$ and $\mathbf{n}_{il}$ are the corresponding normal vectors w.r.t. the edges of the control volumes $b_i$, $b_j$ and $b_l$ corresponding to the edges $e_i,e_j,e_l\in E_h(K)$. Comparing the terms of (\ref{eq:nomsymmetry1}) and (\ref{eq:nomsymmetry2}) we see that in the last term of each equation the integral is over the same edge, but in the first term the integral of the coefficient is over different edges. Since $\alpha$ may be arbitrarily different at those edges, the first terms of (\ref{eq:nomsymmetry1}) and (\ref{eq:nomsymmetry2}) will also be arbitrarily different at thus in general we will have that 
% From this we easily see that in general for varying coefficients the first term of (\ref{eq:nomsymmetry1}) and (\ref{eq:nomsymmetry2}) may be arbitrarily different
% $$\int_{\partial(b_i\cap b_l)}\alpha(s)\;ds\neq\int_{\partial(b_j\cap b_l)}\alpha(s)\;ds$$ 
% and thus in general we have
$$a_h^{FV}(\phi_i,I_h^*\phi_j)\neq a_h^{FV}(\phi_j,I_h^*\phi_i).$$
This completes the proof.
\end{proof}
 
The next lemma is crucial for the analysis of our method. It relates the CRFVE and CRFE bilinear forms.
\begin{lem} For the bilinear forms $a_h^{FE}(u,v)$ and $a_h^{FV}(u,v)$ the following estimates holds
\label{lem:conv}
 \begin{eqnarray}
 \label{eq:difffefve}
  |a_h^{FE}(u,v)-a_h^{FV}(u,I_h^*v)|\leq Ch\|u\|_a\|v\|_a,\qquad \forall u,v\in V_h.
 \end{eqnarray}
 and 
  \begin{equation}
 \label{eq:boundedness}
  a_h^{FV}(u,I_h^*u)\leq C_1\|u\|_{a}\|v\|_{a}
 \end{equation}
 \begin{equation}
 \label{eq:ellipticity}
  a_h^{FV}(u,I_h^*u)\geq C_0\|u\|_{a}^2
 \end{equation}
where $C,C_0,C_1$ are positive constants independent of $h$.

\end{lem}
\begin{proof} Similar results  can be found in \cite{rui2008convergence,chou2003domain} in the case of standard FVE method.
 For all $\alpha(x)\in W^{1,\infty}(K)$, define
 $$\overline{\alpha}_K=\frac{1}{|K|}\int_K\alpha(x)\;dx,\qquad K\in \mathcal{T}_h$$
 and for all $u,v\in V_h$ define
 $$\overline{a}(u,v)=\sum_{K\in T_h}\int_K\overline{\alpha}_K\nabla u\cdot\nabla v\;dx,$$ and
 $$\overline{a}_h(u,I_h^*v)=-\sum_{e\in \mathrm{E}_h^{in}}v(m_e)\int_{\partial b_e}\overline{\alpha}_K\nabla u\cdot\mathbf{n}\;ds.$$
 Since $\overline{\alpha}_K$ is piecewise constant we have from Lemma \ref{lem:fveeqfe}
 $$\overline{a}(u,v)=\overline{a}_h(u,I_h^*v),$$ which gives us
 \begin{eqnarray*}
  a_h^{FE}(u,v)-a_h^{FV}(u,I_h^*v)&=&[a_h^{FE}(u,v)-\overline{a}(u,v)]+[\overline{a}_h(u,I_h^*v)-a_h^{FV}(u,I_h^*v)]\\
  &=&\mathrm{I}+\mathrm{II}.
 \end{eqnarray*}
 Since $\nabla u$ and $\nabla v$ are constant over each element $K$, we have
 $$\mathrm{I}=0.$$
Write $\mathrm{II}$ as
$$  \mathrm{II}=\sum_{e\in \mathrm{E}_h^{in}}v(m_e)\int_{\partial b_e}(\alpha(s)-\overline{\alpha}_K)\nabla u\cdot\mathbf{n}\;ds$$
Define $\gamma_{el}=\partial b_e\cap\partial b_l$. The Cauchy-Schwarz inequality and Bramble-Hilbert give us
 \begin{eqnarray*}
  |\mathrm{II}| &=&\left|\sum_{K\in\mathcal{T}_h}\sum_{e,l\in \mathrm{E}_h(K)}(v(m_e)-v(m_l)\int_{\gamma_{el}}(\alpha(s)-\overline{\alpha}_K)\nabla u\cdot\mathbf{n}_{\gamma_{el}}\;ds\right|\\
   &\leq&\sum_{K\in\mathcal{T}_h}\sum_{e,l\in \mathrm{E}_h(K)}\|(\alpha(s)-\overline{\alpha}_K)\nabla u\|_{0,\infty,K} h_k|(v(m_e)-v(m_l)|\\
    &\leq&C\left(\sum_{K\in\mathcal{T}_h}\sum_{e,l\in \mathrm{E}_h(K)}\|(\alpha(s)-\overline{\alpha}_K)\nabla u\|_{0,\infty,K}^2 h_k^2\right)^{1/2}\left(\sum_{K\in\mathcal{T}_h}\sum_{e,l\in \mathrm{E}_h(K)}|v(m_e)-v(m_l)|^2\right)^{1/2}\\
    &\leq&C\left(\sum_{K\in\mathcal{T}_h}C^2 h_K^2|\alpha(s)|^2_{1,\infty,K}\|\nabla u\|_{0,K}^2\right)^{1/2}\left(\sum_{K\in\mathcal{T}_h}\sum_{e,l\in \mathrm{E}_h(K)}|v(m_e)-v(m_l)|^2\right)^{1/2}\\
%     &\leq&ch|u|_{1,h}|v|_{1,h}\leq ch\|u\|_{1,h}\|v\|_{1,h}\leq ch\|u\|_a\|v\|_a.
    &\leq&Ch|u|_{1,h}|v|_{1,h}\leq Ch\|u\|_a\|v\|_a.
 \end{eqnarray*}
Above we have used the shape regular and quasi-uniform property of the triangulation and the fact that $\alpha\geq1$ and $|\alpha(x)|_{1,\infty,K}$ is uniformly bounded over $\Omega$. The estimates (\ref{eq:boundedness}) and (\ref{eq:ellipticity}) then follow directly from (\ref{eq:difffefve}), cf. \cite{Marcinkowski:2014:ASMFVE} for details.
\end{proof}
 
If we  define for $u,v\in V_h$
\begin{eqnarray}
\label{eq:symnonsymerr}
 a_h^{FV}(u,I_h^*v)=a_h^{FE}(u,v)+\mathrm{E}_h(u,v)
\end{eqnarray}
then, in the the proof of  Lemma~\ref{lem:conv}, we see that there exists a constant independent of $h$, such that
\begin{equation}
\label{eq:symminusnonerror}
\mathrm{E}_h(u,v)\leq Ch\|u\|_{1,h}\|v\|_{1,h}.
\end{equation}
For the CRFVE solution, $u_h^{FV}$, we also have
\begin{eqnarray}
\label{eq:symrhserr}
 a_h^{FE}(u_h^{FV},v)=(f,I_h^*v)-\mathrm{E}_h(u_h^{FV},v).
\end{eqnarray}
The lemma above and the resulting properties are crucial in the analysis of our additive Schwarz method. By applying them and using the framework developed in \cite{ewing2002accuracy} we are able to prove the  $H^1$ error estimates formulated in the following  theorem:
\begin{thm}
\label{thm:h1errorestimate}
 For an exact solution $u\in H^{1+\beta}(\Omega)$ of (\ref{eq:weakformulation}), with $1/2<\beta\leq1$, $f\in L^2(\Omega)$,  $\alpha(x)\in W^{1,\infty}(K)$ and for the CRFVE solution $u_h^{FV}$, we have
 \begin{eqnarray}
  \|u-u_h^{FV}\|_{1,h}\leq Ch^\beta(\|f\|_0+\|u\|_{1+\beta}),
 \end{eqnarray}
 where the constant $C=C(\alpha)$ is independent of $h$.
 \end{thm}
 \begin{proof} A similar proof is given in \cite{rui2008convergence,ewing2002accuracy}.
 
  Let $I_hu\in V_h$ be the CRFE interpolant of $u$ and let $I^*_hu\in V_h^*$ be the CRFVE interpolant of $u$. We start the proof by estimating $\|u_h^{FV}-I_hu\|_{1,h}$. From the coercivity property (\ref{eq:ellipticity}) we have
  
  \begin{align}
  \label{eq:h1part1}
   C_0\|u_h^{FV}-I_hu\|^2_{1,h}&\leq a_h^{FV}(u_h^{FV}-I_hu,I_h^*(u_h^{FV}-I_hu))\nonumber\\
   &=a_h^{FV}(u_h^{FV},I_h^*(u_h^{FV}-I_hu))-a_h^{FV}(I_hu,I_h^*(u_h^{FV}-I_hu))\nonumber\\
   &=(f,I_h^*(u_h^{FV}-I_hu))-a_h^{FE}(u_h^{FE},u_h^{FV}-I_hu)\nonumber\\
   &-a_h^{FE}(I_hu-u_h^{FE},u_h^{FV}-I_hu)-\mathrm{E}_h(I_hu,u_h^{FV}-I_hu).
   \end{align}
In the equations above we have used (\ref{eq:symrhserr}) and (\ref{eq:symnonsymerr}). For clarity of presentation we will split equation (\ref{eq:h1part1}) into three parts and estimate each part independently. Using (\ref{eq:dcrprb2}) and Lemma 5.1 in \cite{Chatzipantelidis:2002:FVE} the two first terms of (\ref{eq:h1part1}) may be estimated by
\begin{align}
 (f,I_h^*(u_h^{FV}-I_hu))-a_h^{FE}(u_h^{FE},u_h^{FV}-I_hu)&=(f,I_h^*(u_h^{FV}-I_hu)-(u_h^{FV}-I_hu))\nonumber\\
 &\leq Ch\|f\|_0\|u_h^{FV}-I_hu\|_{1,h}.\nonumber
\end{align}
From approximation theory, cf. \cite{brenner2008mathematical}, we have that 
\begin{eqnarray}
\label{eq:approxtheory}
 \|u-I_hu\|_{1,h}&\leq& Ch^\beta\|u\|_{1+\beta}\\
 \label{eq:approxtheory2}
 \|I_hu\|_{1,h}&\leq& C\|u\|_{1+\beta}
\end{eqnarray}
which together with the continuity of the finite element bilinear form let us bound the second last remaining term by
\begin{align}
 a_h^{FE}(I_hu-u_h^{FE},u_h^{FV}-I_hu)&\leq C\|I_hu-u_h^{FE}\|_{1,h}\|u_h^{FV}-I_hu\|_{1,h}\nonumber\\
 &\leq C\left(\|I_hu-u\|_{1,h}+\|u-u_h^{FE}\|_{1,h}\right)\|u_h^{FV}-I_hu\|_{1,h}\nonumber\\
 &\leq Ch^\beta\|u\|_{1+\beta}\|u_h^{FV}-I_hu\|_{1,h}.\nonumber
\end{align}
In the second line above we have used the finite element error estimate given below \cite{brenner2008mathematical}
\begin{equation*}
 \|u-u_h^{FE}\|_{1,h}\leq Ch^\beta\|u\|_{1+\beta}.
\end{equation*}
The last term follows straightforwardly from (\ref{eq:symminusnonerror})
\begin{align}
 \mathrm{E}_h(I_hu,u_h^{FV}-I_hu)&\leq Ch\|I_hu\|_{1,h}\|u_h^{FV}-I_hu\|_{1,h}\nonumber\\
 &\leq Ch\|u\|_{1+\beta}\|u_h^{FV}-I_hu\|_{1,h}.\nonumber
\end{align}

% , (\ref{eq:symminusnonerror}) and the $L^2$-stability and -approximation property of $I_h^*$ we may write
% 
% 
% 
%    \begin{align}
%    &\leq(f,I_h^*(u_h^{FV}-I_hu)-(u_h^{FV}-I_hu))\nonumber\\
%    &+C\|I_hu-u\|_{1,h}\|u_h^{FV}-I_hu\|_{1,h}+Ch\|I_hu\|_{1,h}\|u_h^{FV}-I_hu\|_{1,h}\nonumber\\
%    &\leq Ch\|f\|_0\|u_h^{FV}-I_hu\|_{1,h}+Ch\|u\|_{1+\beta}\|u_h^{FV}-I_hu\|_{1,h}.\nonumber\\
%    &\leq Ch(\|f\|_0+\|u\|_{1+\beta})\|u_h^{FV}-I_hu\|_{1,h}.
%   \end{align}
Now, combining the estimates above with the results from approximation theory (\ref{eq:approxtheory})--(\ref{eq:approxtheory2}), we get
\begin{align}
 \|u-u_h^{FV}\|_{1,h}&=\|u-I_hu-(u_h^{FV}-I_hu)\|_{1,h}\nonumber\\
 &\leq\|u-I_hu\|_{1,h}+\|u_h^{FV}-I_hu\|_{1,h}\nonumber\\
 &\leq Ch^\beta\|u\|_{1+\beta}+Ch^\beta(\|f\|_0+\|u\|_{1+\beta}).
\end{align}

% $$\|u-u_h^{FV}\|_{1,h}=\|u-I_hu-(u_h^{FV}-I_hu)\|_{1,h}\leq\|u-I_hu\|_{1,h}+\|u_h^{FV}-I_hu\|_{1,h}\leq Ch^\beta\|u\|_{1+\beta}+Ch^\beta(\|f\|_0+\|u\|_{1+\beta}).$$
This completes the proof.
 \end{proof}
 
% \begin{rmrk}
%  With a slight modification of the proof, we may easily losen the requirement for the the regularity of the exact solution $u$ in Theorem~\ref{thm:h1errorestimate} so that  $0<\beta\leq 1$.
% \end{rmrk}

The main idea in the above proof is motivated by \cite{rui2008convergence,ewing2002accuracy} which in turn was motivated by \cite{chatzipantelidis1999finite}. One of the advantage is that the estimate for $\|u_h^{FV}-I_hu\|_{1,h}$ is not needed, and the approach is more direct and simpler and allows us to apply standard CR finite element error estimation techniques.
\section{The GMRES Method}
\label{sect:gmres}
The linear system of equations which arises from problem (\ref{eq:dcrprb1}) is in general non-symmetric. A popular method for solving such systems is the preconditioned GMRES method; cf. Saad and Schultz \cite{saad1986gmres} and Eistenstat, Elman and Schultz \cite{eisenstat1983variational}. This method has proven to be quite powerful for a large class of non-symmetric problems. The theory originally developed for $L^2(\Omega)$ in \cite{eisenstat1983variational} can easily be extended to an arbitrary Hilbert space; cf. \cite{cai1989some}, see also \cite{cai1992domain}.

We will in this paper use GMRES to solve the linear system of equations
\begin{equation}
 Tu=g,
\end{equation}
where $T$ is a non-symmetric, nonsingular operator, $g\in V_h$ is the right hand side and $u\in V_h$ is the solution vector.

The core of the GMRES method is to solve a least square problem in each iteration, i.e. at step $m$ we approximate the exact solution $u^*=T^{-1}g$ by a vector $u_m\in \mathcal{K}_m$ which minimizes the norm of the residual, where $\mathcal{K}_m$ is the  $m$-th Krylov subspace defined as 
$$\mathcal{K}_m=\text{span}\left\{r_0,Tr_0,\cdots T^{m-1}r_0\right\}$$ and $r_0=g-Tu_0$.
In other words, $z_m$ solves 
\begin{equation*}
 \min_{z\in\mathcal{K}_m}\|g-T(u_0+z)\|_a.
\end{equation*}
Hence, the $m$-th iterate is $u_m=u_0+z_m$.

The convergence rate of the GMRES method is usually expressed in terms of the following two parameters
\begin{equation*}
 c_p=\inf_{u\neq0}\frac{a(Tu,u)}{\|u\|_a^2}\text{ and }C_p=\sup_{u\neq0}\frac{\|Tu\|_a}{\|u\|_a},
\end{equation*}
where $c_p$ corresponds to the smallest eigenvalue of  $\frac{1}{2}(T^t+T)$ the symmetric part of $T$ and $C_p$ corresponds to the square root of the largest eigenvalue of $T^tT$. Here $T^t$ is the transpose of $T$ with respect to the inner product $a(\cdot,\cdot)$. 

The main results regarding the convergence of the GMRES method is stated in the next theorem. It describes the decrease of the norm of the residual in a single step.
\begin{thm}[Eisenstat-Elman-Schultz]
\label{thm:gmres}
If $c_p>0$, then the GMRES method converges and after m steps, the norm of the residual is bounded by
\begin{equation}
 \|r_m\|_a\leq\left(1-\frac{c_p^2}{C_p^2}\right)^{m/2}\|r_0\|_a,
\end{equation}
where $r_m=g-Tu_m$.
\end{thm}
In the next section we will in Theorem \ref{thm:main2} estimate the two parameters describing the convergence rate of the GMRES method once the proposed domain decomposition preconditioner corresponding to the operator $T$ is defined and analyzed.
\section{An Additive Average Method}
\label{sect:asm}
In this section we introduce the additive method for the discrete problem (\ref{eq:dcrprb1}) and provide bounds on the convergence rate, both for the solution of the symmetric and non-symmetric problem.
\subsection{Decomposition of $V_h(\Omega)$}
We decompose the original space into
\begin{eqnarray}
 V_h(\Omega)=V_0(\Omega)+V_1(\Omega)+\cdots+V_N(\Omega),
\end{eqnarray}
where for $i=1,\ldots,N$ we have defined $V_i(\Omega)$ as the restriction of $V_h(\Omega)$ to $\Omega_i$ with functions vanishing on $\partial \Omega^{CR}_{ih}$ and as well as on the other subdomains. The coarse space $V_0(\Omega)$ is defined as the range of the interpolation operator $I_A$. For $u\in V_h(\Omega)$, we let $I_Au\in V_h(\Omega)$ be defined as
\begin{equation}
I_Au:=\begin{cases} u(x), \qquad x\in\partial\Omega_{ih}^{CR}\\ \hat{u}_i,\qquad\quad x\in\Omega_{ih}^{CR} \end{cases} 
\end{equation}
where 
\begin{equation}
\hat{u}_i:=\frac{1}{n_i}\sum_{x\in\partial\Omega_{ih}^{\mathrm{CR}}}u(x).
\end{equation}
Here $n_i$ is the number of nodal points of $\partial\Omega_{ih}^{CR}$.

We also assume that $\mathcal{T}_h(\Omega_i)$ inherits the shape regular and quasi-uniform triangulation for each $\Omega_i$ with mesh parameters $h_i$ and $H_i=diam(\Omega_i)$. The layer along $\partial\Omega_i$ consisting of unions of triangles $K\in\mathcal{T}(\Omega_i)$ which touch $\partial\Omega_i$ is denoted as $\Omega_i^\delta$.

The local bilinear form is chosen as the CRFE symmetric bilinear form $a_h^{FE}(u,v)$ or as the non-symmetric CRFVE bilinear form $a_h^{FV}(u,v)$.

For $i=0,\cdots,N$ we define the projection like operators $T_i\colon V_h\rightarrow V_i$ as
\begin{equation}
 a_h^{FE}(T^{(1)}_iu,v)=a_h^{FE}(u,v)
 \qquad\forall v\in V_i(\Omega),
\end{equation}
for the symmetric problem (\ref{eq:dcrprb2}). For the non-symmetric problem (\ref{eq:dcrprb1}) we introduce two similar projection like operators. The first one which is symmetric is defined as
\begin{equation}
 a_h^{FE}(T^{(2)}_iu,v)=a_h^{FV}(u,I_h^*v)\qquad \forall v\in V_i(\Omega),
\end{equation}
and the second one which is non-symmetric is defined as
\begin{equation}
 a_h^{FV}(T^{(3)}_iu,v)=a_h^{FV}(u,I_h^*v)\qquad \forall v\in V_i(\Omega).
\end{equation}
Each of these problems have a unique solution. We now introduce 
\begin{equation}
 T^{(k)}_A:=T^{(k)}_0+T^{(k)}_1+\cdots+T^{(k)}_N,\qquad k=1,2,3
\end{equation}
which allow us to replace the original problem (\ref{eq:dcrprb1}) for $k=1$  or (\ref{eq:dcrprb2}) for $k=2,3$  by the equation 
\begin{equation}
\label{eq:prcndsystem}
 T^{(k)}_A u=g^{(k)},
\end{equation}
where $g^{(k)}=\sum_{i=0}^Ng_i$ and $g^{(k)}_i=T^{(k)}_iu$. Note that $g^{(k)}_i$ may be computed without knowing the solution
 $u$ of (\ref{eq:dcrprb1}) or (\ref{eq:dcrprb2}), respectively.
% Each of these problems have a unique solution. We now introduce 
% \begin{equation}
%  T^{(k)}_A:=T^{(k)}_0+T^{(k)}_1+\cdots+T^{(k)}_N
% \end{equation}
% which allow us to replace the original symmetric and non-symmetric problems, respectively for $k=1$ and $k=2$,  by the equation 
% \begin{equation}
% \label{eq:prcndsystem}
%  T^{(k)}_Au=g,
% \end{equation}
% where $g=\sum_{i=0}^Ng_i$ and $g_i=T^{(k)}_iu$ is the solution of.
% 
% \begin{equation*}
%  a_h^{FE}(g_i,v)=\left(f,I_h^*v\right)\qquad\forall v\in V_i.
% \end{equation*}
% $a_h^{FE}(\cdot,\cdot)$ above is defined as the restriction of $a_h^{FE}(\cdot,\cdot)$ to $\Omega_i$.
\subsection{Analysis}
\label{sect:analysis}
Let $V_h^{quad}(\Omega_i)$ be the space of continuous piecewise quadratic functions on $T_h(\Omega_i)$.
We introduce a local equivalence mapping $\mathcal{M}_i:V_h(\Omega_i)\rightarrow  V_h^{quad}(\Omega_i)$ in a similar way as in \cite{Brenner:1996:TLS}. 
Let $m_x$ be an adjacent midpoint of a vertex $x$ if both points belong to the same edge in $T_h(\Omega_i)$. The choice of the midpoint is not unique and this fact will be used below.
Note that the degrees of freedom of  $V_h^{quad}(\Omega_i)$ is the sum of $\bar\Omega_{ih}^{CR}$ and $x\in\bar\Omega_{ih}$.
\begin{df}
 For $u\in V_h(\Omega_i)$,
 \begin{equation}
\mathcal{M}_iu(m)=\begin{cases} u(m),&m\in\bar\Omega_{ih}^{CR}, \\
u(m_x)&x\in\bar\Omega_{ih} 
\end{cases}
 \end{equation}
% Here, $\mathcal{T}_x$ is the set of elements sharing the common vertex x. $n_x$ denotes the number of such elements. $x_l$ and $x_r$ are the left- and right neighboring edge midpoints of x, respectively.
\end{df}
The properties of such equivalence mapping, which we are going to use later, are given in the following lemma. %The proofs are given in \cite{marcinkowski1999mortar,sarkis1997nonstandard}.
\begin{lem}
\label{lem:eqvmp}
 Let $\mathcal{M}_i:V_h(\Omega_i)\rightarrow V_h^{quad}(\Omega_i)$ be the local equivalence mapping defined above. The adjacent midpoint $m_x$ is picked as the one whose distant to $\partial\Omega_i$ is the smallest, in particular if $x\in\partial\Omega_{ih}$ then the adjacent midpoint 
is in $\partial\Omega_{ih}^{CR}$.

Then, for any $ u\in V_h{(\Omega_i)}$ we have
 \begin{eqnarray}
  |u|_{1,h,\Omega_i} \leq |\mathcal{M}_iu|_{1,\Omega_i}&\leq&C|u|_{1,h,\Omega_i} ,\\
  \|u-\mathcal{M}_iu\|_{0,\Omega_i}&\leq&Ch_i|u|_{1,h,\Omega_i},\\
  \label{eq:H1bnd-est}
 |\mathcal{M}_iu |_{1,\partial\Omega_i}^2& \leq& C h_i^{-1} |u|_{1,h,\Omega^\delta_i}^2
%   \|u-\mathcal{M}_iu\|_{0,\partial\Omega_i}&\leq&ch_i^{1/2}|u|_{1,h,\Omega_i},\\
%  \int_{\partial\Omega_i}\mathcal{M}_iu(s)\,ds&=&\int_{\partial\Omega_i}u(s)\,ds.
 \end{eqnarray}
Here $\Omega^\delta_{i}$ is the sum of all triangles $K \in T_h(\Omega_i)$ such that $K$ has an edge or a vertex on $\partial\Omega_i$.
\end{lem}
\begin{proof}
The first two statements can be proven in the same way as in \cite{Brenner:1996:TLS}. 

We will prove the last one only.
\begin{eqnarray*}
 |\mathcal{M}_iu |_{1,\partial\Omega_i}^2=\sum_{e\in E_h(\partial\Omega_i)} 
  |\mathcal{M}_iu |_{1,e}^2 \leq C \sum_{e\in E_h(\partial\Omega_i)} \sum_{x\in \partial e}
  \frac{1}{|e|}|\mathcal{M}_iu(x)- \mathcal{M}_iu(m_e)|^2 
\end{eqnarray*}
where $m_e\in\partial\Omega_{ih}^{CR}$ is  the midpoint of an edge $e$.

Note that by the definition  of $\mathcal{M}_iu$ we get that $\mathcal{M}_iu(x)=\mathcal{M}_iu(m_x)$ where $m_x$ is the adjacent midpoint in $\partial\Omega_{ih}^{CR}$, i.e. its left or right neighbor point. 

Thus  by the quasiuniformity of the triangulation and the definition of the equivalence mapping  we get 
\begin{eqnarray*}
 |\mathcal{M}_iu |_{1,\partial\Omega_i}^2\leq \frac{1}{h_i} \sum_{m,s\in\partial\Omega^{CR}_{ih}}
  |\mathcal{M}_iu(m)- \mathcal{M}_iu(s)|^2 =\frac{1}{h_i}\sum_{m,s\in\partial\Omega^{CR}_{ih}}
  |u(m)- u(s)|^2
\end{eqnarray*}
where $m$ and $s$ are neighboring CR points on $\partial\Omega_i$.
Let $x\in \partial\Omega_{ih} $ denote the vertex lying between them, and let 
$\{m_{x,k}\}\subset \Omega_{k,h}^{CR}$ be adjacent midpoints numbered in such a way that 
two successive ones are in one closed element.  
Then from the shape regularity of the triangulation the number of those
midpoints is bounded and a triangle inequality yields that  
$$
  |u(m)-u(s)| \leq |u(m)-u(m_1)|+|u(m_1)-u(m_2)| \ldots +|u(m_k) -u(s)|
$$
Thus, using this and Lemma~\ref{lem:nrmeq} yields that 
$$
  |u(m)-u(s)|^2 \leq C \sum_{x\in \partial K} |u|_{H^1(K)}^2. 
$$
where the sum is taken over all elements $K$ in $\Omega_k$ which has $x$ as a vertex.  

Summing the above estimates over all edges yields the following bound:
\begin{eqnarray*}
 |\mathcal{M}_iu |_{1,\partial\Omega_i}^2
% % % % % % % \leq  Ch_i^{-1} \sum_{x \in \partial\Omega_{ih}}\sum_{\tau:x\in\partial K} |u|_{1,K}^2
\leq C  h_i^{-1} |u|_{1,h,\Omega^\delta_{i}}^2.
\end{eqnarray*}
\end{proof}
 
We are now ready to prove two lemmas for the interpolation-like operator $I_A$ which will help us analyze and prove the main theorems of our proposed method.
 
\begin{lem}
 \label{lem:stability1}
 For any $u\in V_h$ the following holds:
 \begin{equation}
 \label{eq:stb1}
  a_h^{FE}(I_Au,I_Au)\leq C\max_i\left(\frac{\overline\alpha_i}{\underline{\alpha}_i}\frac{H^2_i}{h^2_i}\right)a_h^{FE}(u,u),
 \end{equation}
 where $\overline\alpha_i:=\sup\limits_{x\in{\bar\Omega_i^\delta}}\alpha(x)$, $\underline\alpha_i:=\inf\limits_{x\in{\bar\Omega^\delta_i}}\alpha(x)$ and $C$ is a positive constant independent of $\alpha$,$\frac{\overline\alpha_i}{\underline{\alpha}_i},H_i$ and $h_i$.
\end{lem}
\begin{proof}
The idea behind the proof comes from \cite{dryja2010additive}. We start the proof by estimating  
\begin{eqnarray}
\label{eq:stb11}
\nonumber
\|I_Au\|^2_{a,\Omega_i}&=&\|I_Au\|^2_{a,\Omega_i^\delta}\\ \nonumber
&\leq&\overline\alpha_i|I_Au|^2_{1,h,\Omega_i^\delta}\\ \nonumber
% &=&\sum_{i=1}^N\overline\alpha_i\sum_{K\in \mathcal{T}_h(\Omega_i^h)}|I_Au-\hat{u}_i|^2_{1,h,K}\\ \nonumber
% &\leq&C\sum_{i=1}^N\overline\alpha_i\sum_{K\in \mathcal{T}_h(\Omega_i^h)}h_k^{-2}\|I_Au-\hat{u}_i\|^2_{0,h,K}\\ \nonumber
&\leq&C\overline\alpha_i\sum_{K\in\mathcal{T}_h(\Omega_i^\delta)}\sum_{e,l\in \mathrm{E}_h(K)}(I_Au)(m_e)-(I_Au)(m_l))^2\\
&\leq&C\overline\alpha_i\sum_{x\in \partial\Omega^{CR}_{ih}}(u(x)-\hat{u}_i)^2\nonumber\\ 
&=&C\overline\alpha_i\sum_{x\in \partial\Omega^{CR}_{ih}}(\mathcal{M}_iu(x)-\widehat{\mathcal{M}_iu})^2\nonumber\\ 
% &\leq&C\overline\alpha_i\sum_{x\in \partial\Omega_{ih}}u^2(x)\\ 
&\leq&C\frac{\overline\alpha_i}{h_i}\|\mathcal{M}_iu-\widehat{\mathcal{M}_iu}\|^2_{0,\partial\Omega_i},\nonumber
\end{eqnarray}
Applying the the Poincare inequality and (\ref{eq:H1bnd-est}) of Lemma \ref{lem:eqvmp} we may write 
\begin{eqnarray}
C\frac{\overline\alpha_i}{h_i}\|\mathcal{M}_iu-\widehat{\mathcal{M}_iu}\|^2_{0,\partial\Omega_i}&\leq&C\overline\alpha_i\frac{H_i^2}{h_i}|\mathcal{M}_iu |_{1,\partial\Omega_i}^2\nonumber\\
&\leq& C \left(\overline\alpha_i\frac{H^2_i}{h^2_i}\right) |u|_{1,h,\Omega^\delta_i}^2\nonumber\\
&\leq& C \left(\frac{\overline\alpha_i}{\underline{\alpha}_i}\frac{H^2_i}{h^2_i}\right) \|u\|_{a,\Omega^\delta_i}^2.\nonumber
\end{eqnarray}
% \begin{eqnarray}
%  \|u-\hat{u}_i\|_{0,\Omega_i^h}^2&=& \|u-\mathcal{M}_iu+\mathcal{M}_iu-\hat{u}_i\|_{0,\Omega_i^h}^2\nonumber\\
%  &\leq& 2\left\{\|u-\mathcal{M}_iu\|_{0,\Omega_i^h}^2+\|\mathcal{M}_iu-\hat{u}_i\|_{0,\Omega_i^h}^2\right\}\nonumber\\
%  &\leq& 2\left\{Ch_i^2|u|_{1,h,\Omega_i^h}^2+CH_i^2|\mathcal{M}_iu|_{1,\Omega_i^h}^2\right\}\nonumber\\
%  &\leq& 2\left\{Ch_i^2|u|_{1,h,\Omega_i^h}^2+CH_i^2|u|_{1,h,\Omega_i^h}^2\right\}.
% \end{eqnarray}
% Combining the above estimate with (\ref{eq:stb13}) and the fact that $\frac{H_i}{h_i}>1$ we get
% \begin{eqnarray}
%  C\frac{\overline\alpha_i}{h_i}\|u-\hat{u}_i\|_{0,\partial\Omega_i}^2&\leq&C\overline\alpha_i\frac{H^2_i}{h_i^2}|u|_{1,h,\Omega^h_i}^2\nonumber\\
%  &\leq& C\frac{\overline\alpha_i}{\underline\alpha_i}\frac{H^2_i}{h_i^2}\|u\|_{a,\Omega^h_i}^2.
% \end{eqnarray}
% 
% 
% Using the fact that 
% \begin{equation}
%  |\mathcal{M}_iu|_{1,\partial\Omega_i}^2\leq \frac{1}{\underline\alpha_ih_i}|\mathcal{M}_iu|_{a,\Omega^h_i}^2,
% \end{equation}
% From this we have that 
% \begin{equation}
%  |I_Au|^2_{a,\Omega_i}\leq C\frac{\overline\alpha_i}{\underline\alpha_i}\frac{H^2_i}{h_i^2}|u|_{a,\Omega_i}^2,
% \end{equation}
Summing over all the subdomains and introducing $\max_i\left(\frac{\overline\alpha_i}{\underline{\alpha}_i}\frac{H^2_i}{h^2_i}\right)$ we prove (\ref{eq:stb1}).
\end{proof}
 
% The estimate above is not sharp with respect to $\frac{H_i}{h_i}$ and under certain assumptions on the lower bound of $\alpha(x)$ in the interior of each $\Omega_i$ we may improve the estimate.
Under certain assumptions on the lower bound of $\alpha(x)$ in the interior of each $\Omega_i$ we may improve the above estimate with respect to $\frac{H_i}{h_i}$.
\begin{lem}
 \label{lem:stability2}
 Let $\underline\alpha_i\leq \alpha(x)$ in $\Omega_i\setminus\Omega_i^\delta$. For any $u\in V_h$ the following holds:
 \begin{equation}
  a_h^{FE}(I_Au,I_Au)\leq C\max_i\left(\frac{\overline\alpha_i}{\underline{\alpha}_i}\frac{H_i}{h_i}\right)a_h^{FE}(u,u),
 \end{equation}
 where $\overline\alpha_i:=\sup\limits_{x\in{\bar\Omega_i^\delta}}\alpha(x)$, $\underline\alpha_i:=\inf\limits_{x\in{\bar\Omega^\delta_i}}\alpha(x)$ and $C$ is a positive constant independent of $\alpha$,$\frac{\overline\alpha_i}{\underline{\alpha}_i},H_i$ and $h_i$.
\end{lem}
\begin{proof}
From the proof of Lemma \ref{lem:stability1} we have that
\begin{eqnarray}
\label{eq:stb}
\|I_Au\|^2_{a,\Omega_i}
% &\leq&C\overline\alpha_i\sum_{x\in \partial\Omega^{CR}_{ih}}(u(x)-\hat{u}_i)^2\nonumber\\ 
% &=&C\overline\alpha_i\sum_{x\in \partial\Omega^{CR}_{ih}}(\mathcal{M}_iu(x)-\widehat{\mathcal{M}_iu})^2\nonumber\\ 
% &\leq&C\overline\alpha_i\sum_{x\in \partial\Omega_{ih}}u^2(x)\\ 
&\leq&C\frac{\overline\alpha_i}{h_i}\|\mathcal{M}_iu-\widehat{\mathcal{M}_iu}\|^2_{0,\partial\Omega_i}.\nonumber
\end{eqnarray}
% Where we have used the fact that in (\ref{eq:stb})
% \begin{equation*}
% \sum_{x\in \partial\Omega_{ih}}(u(x)-\hat{u}_i)^2=\sum_{x\in \partial\Omega_{ih}}(u(x)-\hat{u}_i)u(x)=\sum_{x\in \partial\Omega_{ih}}u^2(x)-n_i\hat{u}^2_i\leq\sum_{x\in \partial\Omega_{ih}}u^2(x).
% \end{equation*}
% If we use the local equivalence mapping $\mathcal{M}_i$ we may bound $\|u-\hat{u}_i\|^2_{0,\partial\Omega_i}$ by
% \begin{eqnarray*}
% \|u-\hat{u}_i\|^2_{0,\partial\Omega_i}&=&\|u-\mathcal{M}_iu+\mathcal{M}_iu-\hat{u}_i\|^2_{0,\partial\Omega_i}\\
% &\leq&2\left\{\|u-\mathcal{M}_iu\|^2_{0,\partial\Omega_i}+\|\mathcal{M}_iu-\hat{u}_i\|^2_{0,\partial\Omega_i}\right\}\\
% &\leq& Ch_i|u|^2_{1,h,\Omega_i}+2\|\mathcal{M}_iu-\hat{u}_i\|^2_{0,\partial\Omega_i}.
% \end{eqnarray*}
%Define $w=\mathcal{M}_iu-\widehat{\mathcal{M}_iu}$.
Using a scaling argument and a trace theorem we may write: 
\begin{eqnarray}
\label{eq:stb2}
\|I_Au\|^2_{a,\Omega_i}&\leq&C\frac{\overline\alpha_i}{h_i}\|\mathcal{M}_iu-\widehat{\mathcal{M}_iu}\|^2_{0,\partial\Omega_i}\nonumber\\
&\leq&C\overline\alpha_i\frac{H_i}{h_i}\left\{|\mathcal{M}_iu|^2_{1,h,\Omega_i}+H_i^{-2}\|\mathcal{M}_iu-\widehat{\mathcal{M}_iu}\|^2_{0,\Omega_i}\right\}\\
&\leq&C\overline\alpha_i\frac{H_i}{h_i}|\mathcal{M}_iu|^2_{1,h,\Omega_i}\nonumber\\
% &=&C\overline\alpha_i\frac{H_i}{h_i}\left(|\mathcal{M}_iu|^2_{1,h,\Omega_i^h}+ch_i|u|^2_{1,h,\Omega_i}\right)\nonumber\\
&=&C\overline\alpha_i\frac{H_i}{h_i}|u|^2_{1,h,\Omega_i}\nonumber\\
&\leq&C\frac{\overline\alpha_i}{\underline\alpha_i}\frac{H_i}{h_i}\|u\|^2_{a,\Omega_i}\nonumber
% &=&C\max_i\left(\frac{\overline\alpha_i}{\underline{\alpha}_i}\frac{H_i}{h_i}\right)a_h^{FE}(u,u).\nonumber
\end{eqnarray}
where we have used the properties of $\mathcal{M}_i$, and Poincare's inequality on the last term in the curly brackets. Summing over all the subdomains and introducing $\max_i\left(\frac{\overline\alpha_i}{\underline{\alpha}_i}\frac{H_i}{h_i}\right)$ completes the proof.
\end{proof}
 
Using the two lemmas above we may now state two theorems and two propositions for the convergence rate of our proposed preconditioner applied to the linear system arising from the symmetric problem (\ref{eq:dcrprb2}) and for the linear system arising from the non-symmetric problem (\ref{eq:dcrprb1}). We first prove the convergence rate for our ASM applied to the symmetric problem (\ref{eq:dcrprb2})
\begin{thm}
\label{thm:main1}
 For any $u\in V_h$ the following holds:
 \begin{equation}
 \label{eq:thm1}
  C_1\beta_1^{-1}a_h^{FE}(u,u)\leq a_h^{FE}(T^{(1)}_Au,u)\leq C_2a_h^{FE}(u,u),
 \end{equation}
where $\beta_1=\max_i\left(\frac{\overline\alpha_i}{\underline{\alpha}_i}\frac{H_i^2}{h_i^2}\right)$ and the positive constants $C_1$ and $C_2$ is independent of $\alpha$,$\frac{\overline\alpha_i}{\underline{\alpha}_i},H_i$ and $h_i$ for $i=1,\cdots,N$.
\end{thm}
\begin{proof}
Following the general theory of ASMs, we need to check the three key assumptions (\cite{smith1996domain, toselli2005domain}).
\begin{ass}[1]
For all $u\in V_h$ there exists a representation $u=\sum_{i=0}^N u_i,\; u_i\in V_i$, such that
\begin{equation}
 \label{eq:ass1}
\sum_{i=0}^Na_h^{FE}(u_i,u_i)\leq C\beta_1a_h^{FE}(u,u).
\end{equation}
\end{ass}
Let $u_0=I_A u$ for $u\in V_h(\Omega)$ and $u_i:=u-u_0$ on $\overline\Omega_i$ and $u_i=0$ outside of $\Omega_i$. Obviously $u_i\in V_i(\Omega)$ for $i=0,\ldots,N$, and $u=\sum_{i=0}^Nu_i$.  We then have
\begin{eqnarray}
\sum_{i=1}^N a_h^{FE}(u_i,u_i)+a_h^{FE}(u_0,u_0)&=&\sum_{i=1}^N a_h^{FE}(u-u_0,u-u_0)+a_h^{FE}(u_0,u_0)\nonumber\\
&\leq&2\sum_{i=1}^N \{a_h^{FE}(u,u)+a_h^{FE}(u_0,u_0)\}+a_h^{FE}(u_0,u_0)\nonumber\\
&=&2a_h^{FE}(u,u)+3a_h^{FE}(u_0,u_0).
\end{eqnarray}
Using Lemma \ref{lem:stability1} on the last term we obtain $\beta_1$ in (\ref{eq:ass1}) immediately.
% \begin{equation}
%  b_{0}(u_0,u_0)\leq C\beta_1a_h^{FE}(u,u).
% \end{equation}
% Thus, we have that 
% % \begin{eqnarray}
% % %  \sum_ia_h^{FE}(u_i,u_i)&=& a_h^{FE}(u,u).
% % %  &\leq&2\left\{\\sum_{=1}^Nia_h^{FE}(u_i,u_i)
% % \end{eqnarray}

\begin{ass}[2] Let $0\leq\mathcal{E}_{ij}\leq1$ be the minimal values that satisfy 
\begin{equation*}
 a_h^{FE}(u_i,u_j)\leq\mathcal{E}_{ij}a_h^{FE}(u_i,u_i)^{1/2}a_h^{FE}(u_j,u_j)^{1/2},\qquad\forall u_i\in V,\;\forall u_j\in V_j\;i,j=1,\ldots,N
\end{equation*}
Define $\rho(\mathcal{E})$ to be the spectral radius of $\mathcal{E}=\{\mathcal{E}_{ij}\}$.
\end{ass}
\noindent In our case $V_i$ and $V_j$ are orthogonal for $i\neq j$, thus $\rho(\mathcal{E})=1$.

Since we are using exact bilinear forms the next assumption is trivially satisfied with $\omega=1$ for $i=1,\ldots,N$.
\begin{ass}[3]
Let $\omega>0$ be the minimal constant such that 
\begin{equation*}
 a_h^{FE}(u,u)\leq\omega a_h^{FE}(u,u),\qquad u\in V_i.
\end{equation*}
\end{ass}
% We use exact bilinear forms so in our case $\omega=1$ for $i=1,\ldots,N$.
\end{proof}
 
This results may be improved if the condition on the distribution of $\alpha$ in Lemma~\ref{lem:stability2} is satisfied as shown in the next Proposition.
\begin{prop}
\label{prop:main1}
 Let $\underline\alpha_i\leq \alpha(x)$ in $\Omega_i\setminus\Omega_i^\delta$. For any $u\in V_h$ the following holds:
 \begin{equation}
 \label{eq:prop1}
  C_1\beta_1^{-1}a_h^{FE}(u,u)\leq a(T^{(1)}_Au,u)\leq C_2a_h^{FE}(u,u),
 \end{equation}
where $\beta_1=\max_i\left(\frac{\overline\alpha_i}{\underline{\alpha}_i}\frac{H_i}{h_i}\right)$ and the positive constants $C_1$ and $C_2$ is independent of $\alpha$,$\frac{\overline\alpha_i}{\underline{\alpha}_i},H_i$ and $h_i$ for $i=1,\cdots,N$.
\end{prop}
\begin{proof}
 The proof is completely analogous to Theorem \ref{thm:main1}, but Lemma~\ref{lem:stability2} is applied instead of Lemma~\ref{lem:stability1}.
\end{proof}
 
The main theorem for the GMRES convergence rate of our ASM applied to the non-symmetric problem (\ref{eq:dcrprb1}) is stated below
% \begin{thm}
% \label{thm:main2}
%   For any $u\in V_h$ and if $h<h_1<h_0$ the following holds:
%  \begin{eqnarray}
%   a(T^2_Au,T^2_Au)&\leq &C_2^2a_h^{FE}(u,u),\\
%   a(T^2_Au,u)&\geq& C_1a_h^{FE}(u,u),
%  \end{eqnarray}
% where $C_2=4M$, $C_1=\gamma^2\beta^{-1}_1-4Mch$ and $\beta_1=\max_i\left(\frac{\overline\alpha_i}{\underline{\alpha}_i}\frac{H^2_i}{h^2_i}\right)$.
% % and the positive constants $C_1$ and $C_2$ is independent of $\alpha$,$\frac{\overline\alpha_i}{\underline{\alpha}_i},H_i$ and $h_i$ for $i=1,\cdots,N$.
% \end{thm}
\begin{thm} \label{thm:main2}
 There exists $h_0>0$ such that for all $h<h_0$, $k=2,3,$ and $u\in V_h$, we have
\begin{eqnarray*}
 \|T^{(k)}u\|_a&\leq& C\|u\|_a,  \\
a_h^{FE}(T^{(k)}u,u)&\geq& c \max_i\frac{\overline\alpha_i}{\underline{\alpha}_i}\left(\frac{H_i}{h_i}\right)^{-2}
 \: a_h^{FE}(u,u) ,
\end{eqnarray*}
where $C,c$ are positive constants independent of $\alpha$, $\frac{\overline\alpha_i}{\underline{\alpha}_i}$, $h_i$ and $H_i$ for $i=1,\ldots,N.$
\end{thm}
\begin{proof}
Following the framework of \cite{Marcinkowski:2014:ASMFVE} we need to prove three assumptions.
\begin{ass}[1] For all $u,v\in V_h$ the following holds
  \begin{eqnarray}
  |a_h^{FE}(u,v)-a_h^{FV}(u,I_h^*v)|\leq Ch\|u\|_a\|v\|_a, 
 \end{eqnarray}
\end{ass}
\begin{ass}[2]
 For all $u\in V_h$ there exists a representation $u=\sum_{i=0}^N u_i,\; u_i\in V_i$, such that
\begin{equation}
 \label{eq:ass2}
\sum_{i=0}^Na_h^{FE}(u_i,u_i)\leq C\beta_1a_h^{FE}(u,u).
\end{equation}
\end{ass}
\begin{ass}[3]
 Let $0\leq\mathcal{E}_{ij}\leq1$ be the minimal values that satisfy 
\begin{equation*}
 a_h^{FE}(u_i,u_j)\leq\mathcal{E}_{ij}a_h^{FE}(u_i,u_i)^{1/2}a_h^{FE}(u_j,u_j)^{1/2},\qquad\forall u_i\in V,\;\forall u_j\in V_j\;i,j=1,\ldots,N
\end{equation*}
Define $\rho(\mathcal{E})$ to be the spectral radius of $\mathcal{E}=\{\mathcal{E}_{ij}\}$.
\end{ass}
These assumptions have been proven in Theorem \ref{thm:main1} and Lemma \ref{lem:conv}.

\end{proof}
 
% \begin{corr}
%  Let $\underline\alpha_i\leq \alpha(x)$ in $\Omega_i\setminus\Omega_i^h$. For any $u\in V_h$ and if $h<h_1<h_0$ the following holds:
%  \begin{eqnarray}
%   a(T^2_Au,T^2_Au)&\leq &C_2^2a_h^{FE}(u,u),\\
%   a(T^2_Au,u)&\geq& C_1a_h^{FE}(u,u),
%  \end{eqnarray}
% where $C_2=4M$, $C_1=\gamma^2\beta^{-1}_1-4Mch$ and $\beta_1=\max_i\left(\frac{\overline\alpha_i}{\underline{\alpha}_i}\frac{H_i}{h_i}\right)$.
% \end{corr}

In the same way as for the convergence rate of our ASM applied to the symmetric problem we may improve the estimate of the last theorem if the condition of the distribution of $\alpha$  in Lemma~\ref{lem:stability2} are satisfied.
\begin{prop}
 There exists $h_0>0$ such that for all $h<h_0$, $k=2,3,$, $u\in V_h$ and $\underline\alpha_i\leq \alpha(x)$ in $\Omega_i\setminus\Omega_i^\delta$, we have
\begin{eqnarray*}
 \|T^{(k)}u\|_a&\leq& C\|u\|_a,  \\
a_h^{FE}(T^{(k)}u,u)&\geq& c\max_i \frac{\overline\alpha_i}{\underline{\alpha}_i}\left(\frac{H_i}{h_i}\right)^{-1}
 \: a_h^{FE}(u,u) \qquad \forall u \in V_h,
\end{eqnarray*}
where $C,c$ are positive constants independent of $\alpha$, $\frac{\overline\alpha_i}{\underline{\alpha}_i}$, $h_i$ and $H_i$ for $i=1,\ldots,N.$
\end{prop}
\begin{proof}
 The proof is completely analogous to Theorem \ref{thm:main2}. The only difference is that the assumptions here have been proven in Lemma \ref{lem:conv} and Proposition \ref{prop:main1} instead of in Theorem \ref{thm:main1}.
\end{proof}

\section{Numerical results}
\label{sect:numres}
In this section we present some numerical results using the proposed method. All experiments are done for the Problem (\ref{eq:modelproblem}) on a unit square domain $\Omega=(0,1)^2$. The coefficient $\alpha$ is equal to $2+\sin(100\pi x)\sin(100\pi y)$ except for the areas marked with red where $\alpha$ equals $\alpha_1(2+\sin(100\pi x)\sin(100\pi y))$ and $\alpha_1$ is a parameter describing the discontinuities in the distribution of the coefficient. The right hand side is chosen to be $f=1$.

The numerical solution is obtained by solving the preconditioned system (\ref{eq:prcndsystem}) for $k$ equal 2 using the generalized minimal residual method (GMRES). We run the method until the $l_2$ norm of the residual is reduced by a factor $10^{6}$, i.e., when $\|r_i\|_2/\|r_0\|_2\leq 10^{-6}$, $r_i$ being the $i$-th residual.

In the first four examples we subdivide $\Omega$ into $4\times 4$ subdomains and test the method for various distributions of the coefficient $\alpha$. For example 1, we consider a distribution of $\alpha$ consisting of channels and inclusions in the interior of the subdomains, i.e. $\alpha$ has jumps only in the interior of subdomains (cf. Figure \ref{fig:alphaint}). For Example 2 and 4, we consider distributions where $\alpha$ has jumps along subdomain interfaces (cf. Figure \ref{fig:alphabnd} and \ref{fig:alphachannelboundary}) and therefore jumps also on the subdomain layers. In Example 3, we consider the case where $\alpha$ has jumps over substructures. For each of the examples above, the number of iterations until convergence for different values of $\alpha_1$, are shown in Table \ref{tbl:itnmb}.  

In Table \ref{tbl:asstab1} and \ref{tbl:asstab2} we report the iteration number for decreasing values of $H_i$ and $h_i$ for two test cases where the coefficient $\alpha$ is equal to $2+\sin(10\pi x)\sin(10\pi y)$ and $2+\sin(100\pi x)\sin(100\pi y)$, respectively. In the parentheses we report an estimate of the smallest eigenvalue of the symmetric part of the preconditioned operator $T^{(2)}_A$, i.e., the smallest eigenvalue of $\frac{1}{2}\left({T_A^{(2)}}^t+T_A^{(2)}\right)$, which corresponds to the parameter $c_p$ in Theorem~\ref{thm:gmres}. 

In Table~\ref{tbl:asstabNS}, we report the iteration number  and estimate of the smallest eigenvalue for the symmetric part of the non-symmetric preconditioner for decreasing values of $H_i$ and $h_i$, i.e., for $k$ equal $3$. The distribution of $\alpha$ is here the same as for the problem in Table~\ref{tbl:asstab2}. 

We do not report any estimates of the parameter $C_p$ of Theorem~\ref{thm:gmres}, which is defined as the square root of the largest eigenvalue of the normal matrix, ${T_A^{(2)}}^tT_A^{(2)}$, since both our convergence analysis and numerical results show that this is a constant independent of the coefficient $\alpha$ and the mesh parameters.

The magnitude of the non-symmetry and non-normality of the CRFVE stiffness matrix A with respect to $\alpha_1$  are shown in Table \ref{tbl:nonsymmetry} and the distributions of the eigenvalues of stiffness matrix A and the corresponding preconditioned operator, $T_A^{(2)}$, are shown in Figure \ref{fig:eigenvaluesofA} and \ref{fig:eigenvaluesofprecondsystem}, respectively. The difference between the finite element and the finite volume element stiffness matrices measured in the matrix 2-norm is shown in Table~\ref{tbl:errorfvefe}, for three different distributions of the coefficient $\alpha$.

\noindent
\begin{figure}[htb]
\centering
\begin{subfigure}[t]{0.48\linewidth}
        \centering
         \includegraphics[trim=145 60 145 60, clip,width=\linewidth]{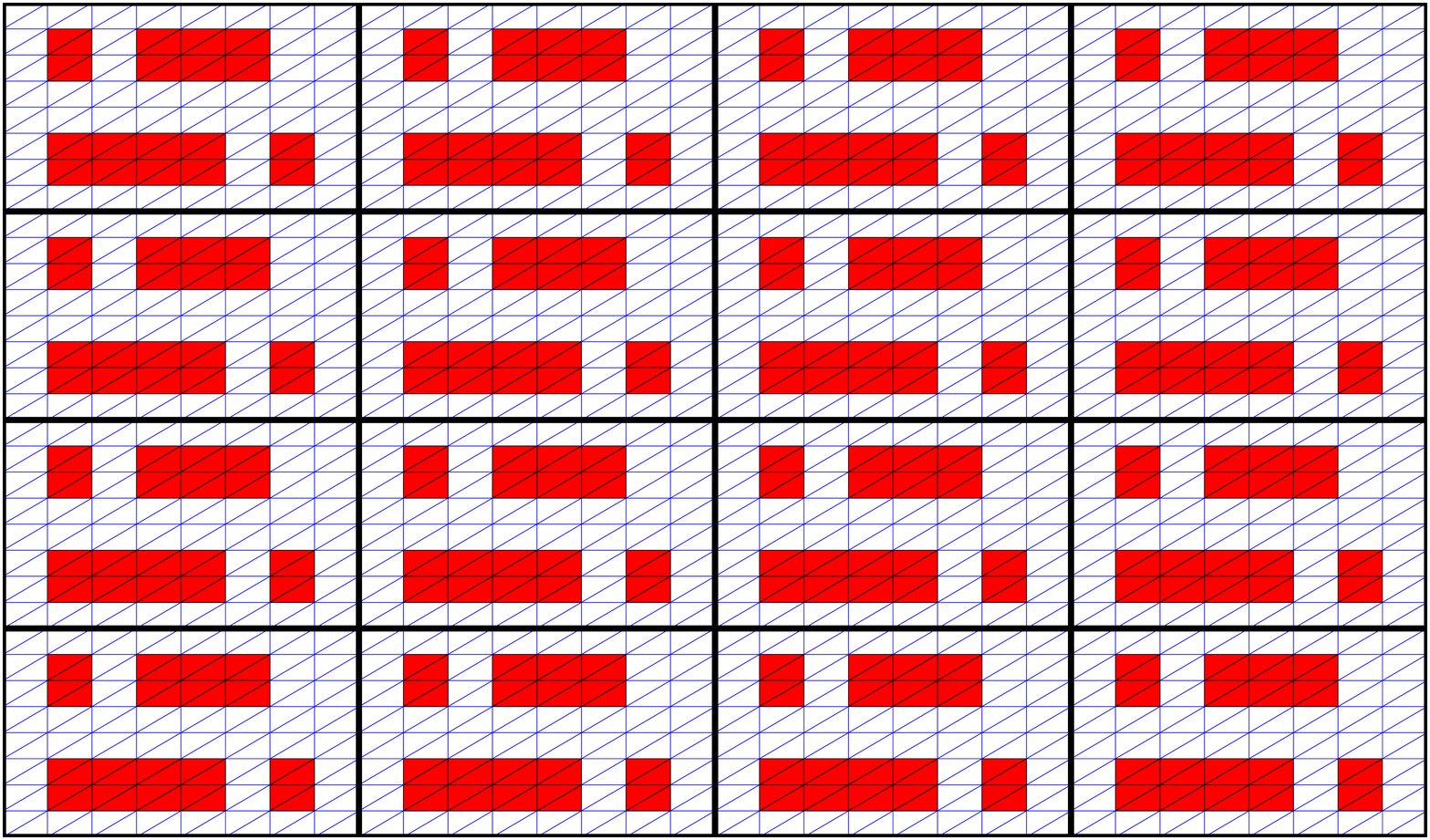}
%  \rule{4cm}{3cm}
  \caption{Example 1.}
 \label{fig:alphaint}

\end{subfigure}
\hspace{0.1cm}
\begin{subfigure}[t]{0.48\linewidth}
        \centering
%  \rule{4cm}{3cm}
 \includegraphics[trim=145 60 145 60, clip,width=\linewidth]{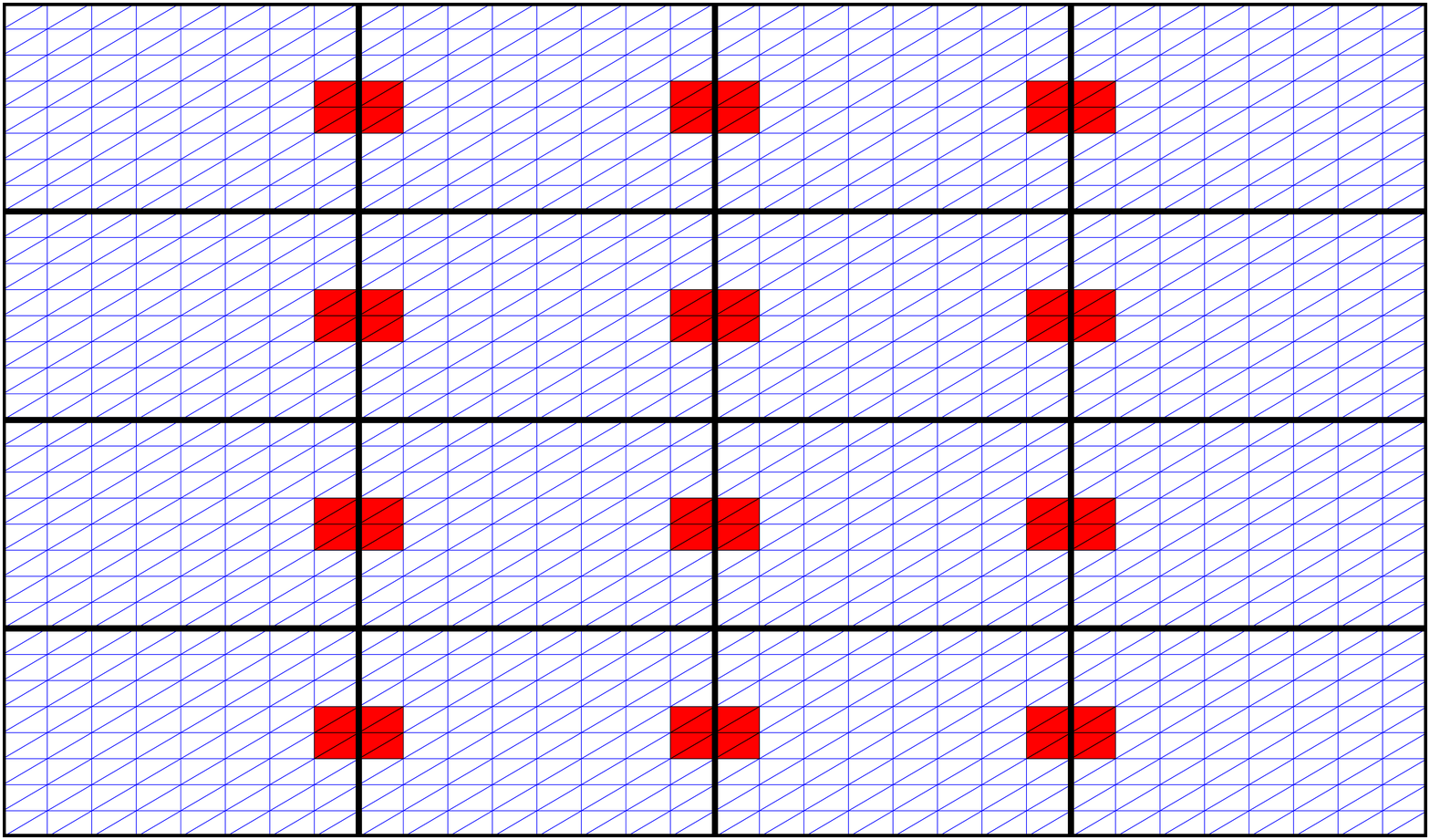}
 \caption{Example 2. }
 \label{fig:alphabnd}
\end{subfigure}
\caption{Two geometries with $32\times32$ fine mesh and $4\times4$ coarse mesh showing the distribution of $\alpha$. The regions marked with red are where $\alpha_1$ has a large value.}
\end{figure}
\begin{figure}[htb]
        \centering
\begin{subfigure}[t]{0.48\linewidth}
         \includegraphics[trim=145 60 145 60, clip,width=\linewidth]{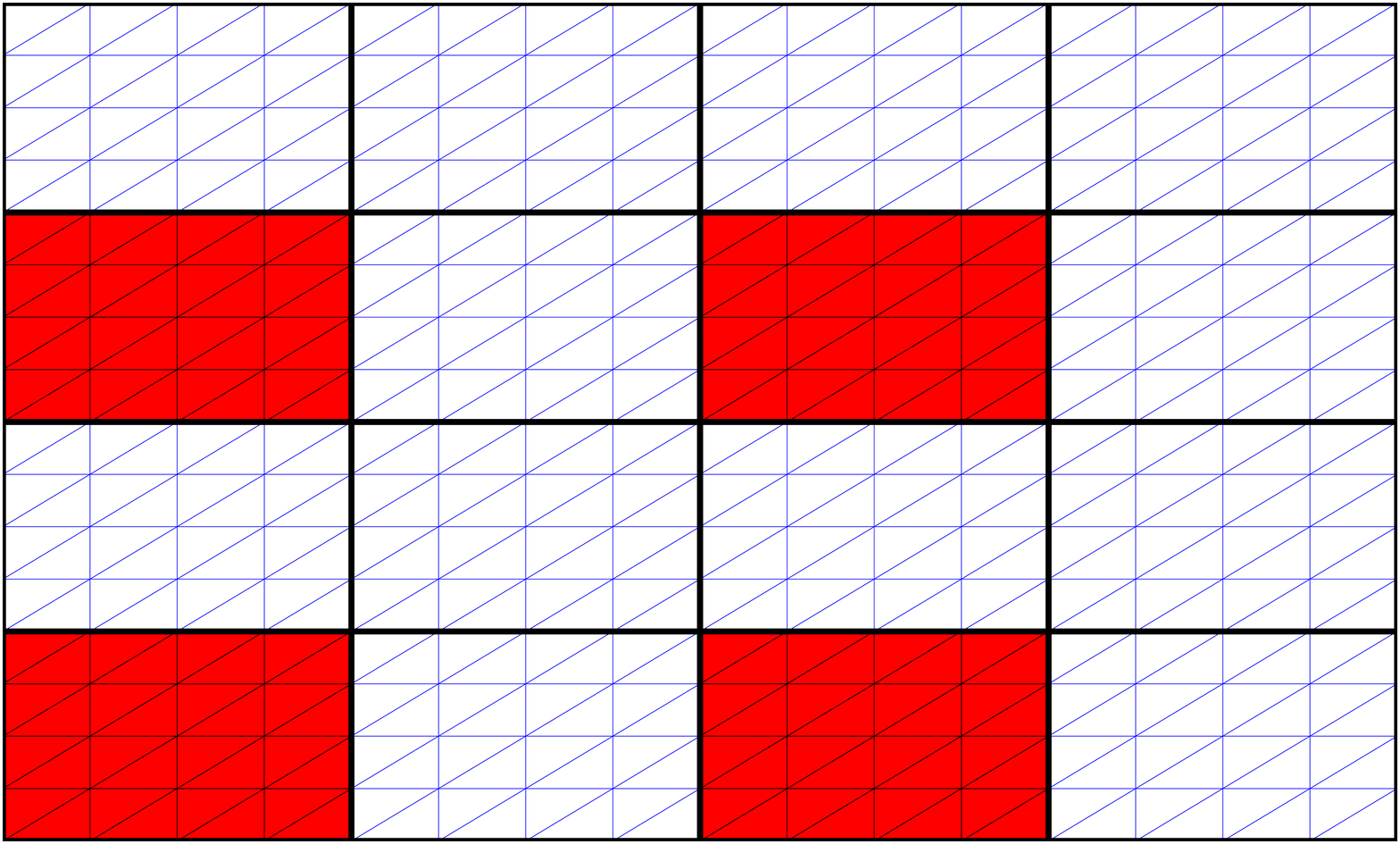}
%  \rule{4cm}{3cm}
  \caption{Example 3.}
 \label{fig:alphajumpsub}
 \end{subfigure}
 \begin{subfigure}[t]{0.48\linewidth}
         \includegraphics[trim=145 60 145 60, clip,width=\linewidth]{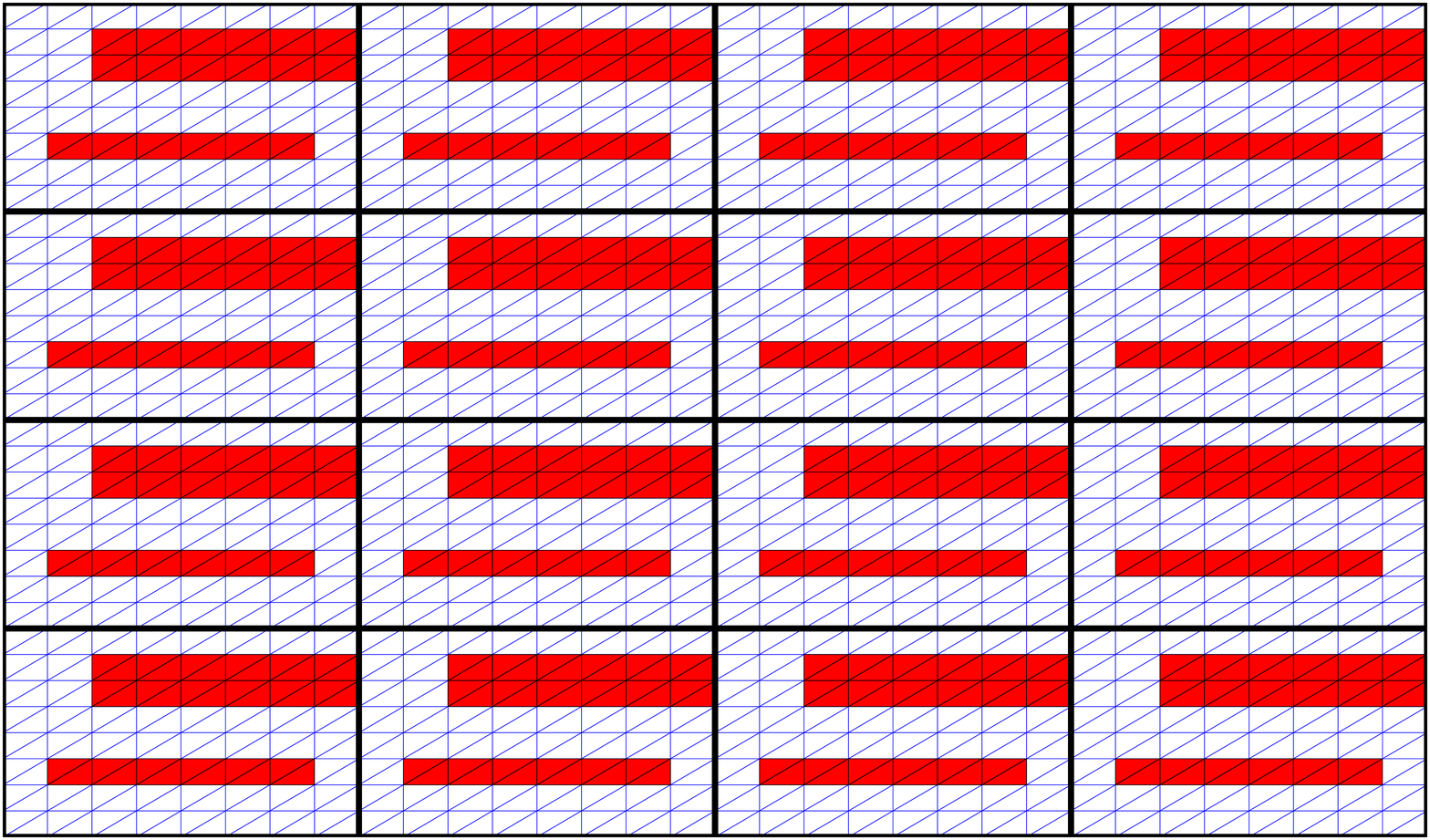}
%  \rule{4cm}{3cm}
  \caption{Example 4.}
 \label{fig:alphachannelboundary}
 \end{subfigure}
\caption{(A) Geometry with $16\times16$ fine mesh and $4\times4$ coarse mesh showing the distribution of $\alpha$ for the third example. (B) Geometry with $32\times32$ fine mesh and $4\times4$ coarse mesh showing the distribution of $\alpha$ for the fourth example. The regions marked with red are where the coefficient has jumps.}
 \end{figure}
\begin{table}[htb]
\centering
\begin{tabular}{|l c c |}
\hline
% &Average ASM&&&\\
% \hline
% $\alpha_1$&Example 1:&Example 2:&Example 3:&Example 4:\\
% \hline
$\alpha_1$ &$\|A-A^t\|_2$&$\|AA^t-A^tA\|_2$\\
% &$\frac{\|A-A^t\|_2}{\|A\|_2}$&$\frac{\|AA^t-A^tA\|_2}{\|A\|_2}$&$\frac{\|AA^t-A^tA\|_2}{\|A\|^2_2}$\\
\hline
%   1     &17(11.41) \\
%   $10^0$&17(6.82)&21(12.20)& 1321\\
  $10^0$&4.0e-1&6.48e0\\
%   &1.64e-2&2.68e-1&1.11e-2\\
  $10^1$&3.96e0&6.03e2\\
%   &1.63e-2&2.61e0&1.12e-2\\
  $10^2$&3.96e1&6.06e4\\
%   &1.63e-2&2.61e1&1.13e-2\\
  $10^3$&3.96e2&6.06e6\\
%   &1.63e-2&2.62e2&1.13e-2\\
  $10^4$&3.96e3&6.06e8\\
%   &1.63e-2&2.62e3&1.13e-2\\
  $10^5$&3.96e4&6.06e10\\
%   &1.63e-2&2.62e4&1.13e-2\\
  $10^6$&3.96e5&6.06e12\\
%   &1.63e-2&2.62e5&1.13e-2\\
\hline
\end{tabular}
\vspace{5mm}
\caption{2-norm measures of the non-symmetry and non-normality of the CRFVE stiffness matrix $A$ with the distribution of $\alpha$ given in Example 4.} 
\label{tbl:nonsymmetry}
\end{table}
\begin{table}[htb]
\centering
\begin{tabular}{|l c c c |}
% &Average ASM&&&\\
% \hline
% $\alpha_1$&Example 1:&Example 2:&Example 3:&Example 4:\\
\hline
&$\|A^{FE}-A^{FVE}\|_2$&$\|A^{FE}-A^{FVE}\|_2$&$\|A^{FE}-A^{FVE}\|_2$\\
\hline
$h$&$\alpha=2+\sin(100\pi x)\sin(100\pi y)$&$\alpha=2+\sin(10\pi x)\sin(10\pi y)$&$\alpha=2+\sin(\pi x)\sin(\pi y)$\\
\hline
1/8& 7.16e-1&4.10e0&5.35e-1\\
1/16& 1.02e-1&2.31e0&2.82e-1\\
1/32&1.52e0&1.16e0&1.44e-1\\
1/64&4.05e0&6.52e-1&7.28e-2\\
1/128&3.16e0&3.47e-1&3.65e-2\\
1/256&1.41e0&1.79e-1&1.84e-2\\
1/512&7.91e-1& 9.09e-2&9.22e-3\\

\hline

\end{tabular}
\vspace{5mm}
\caption{2-norm measures of the difference between the finite element and the finite volume element stiffness matrix for decreasing $h$ for three different distributions of $\alpha$.} 
\label{tbl:errorfvefe}
\end{table}
 
\begin{table}[htb]
\centering
\begin{tabular}{|l l l l l|}
\hline
&&Average ASM&&\\
\hline
&Example 1:&Example 2:&Example 3:&Example 4:\\
\hline
$\alpha_1$ &$\sharp$ iter.&$\sharp$ iter.&$\sharp$ iter.&$\sharp$ iter.\\
\hline
%   1     &17(11.41) \\
%   $10^0$&17(6.82)&21(12.20)& 1321\\
  $10^0$&40&40&31&40\\
  $10^1$&38&66&32&52\\
  $10^2$&37&108&36&92\\
  $10^3$&37&177&36&140\\
  $10^4$&37&233&38&178\\
  $10^5$&37&276&39&214\\
  $10^6$&37&316&39&249\\
\hline
\end{tabular}
\vspace{5mm}
\caption{Number of iterations until convergence for the solution of (\ref{eq:modelproblem}) with different values of $\alpha_1$ in the distributions of the coefficient $\alpha$ given in Figures \ref{fig:alphaint},  \ref{fig:alphabnd}, \ref{fig:alphajumpsub}, \ref{fig:alphachannelboundary}.}
\label{tbl:itnmb}
\end{table}

\begin{figure}
        \centering
         \includegraphics[width=\linewidth]{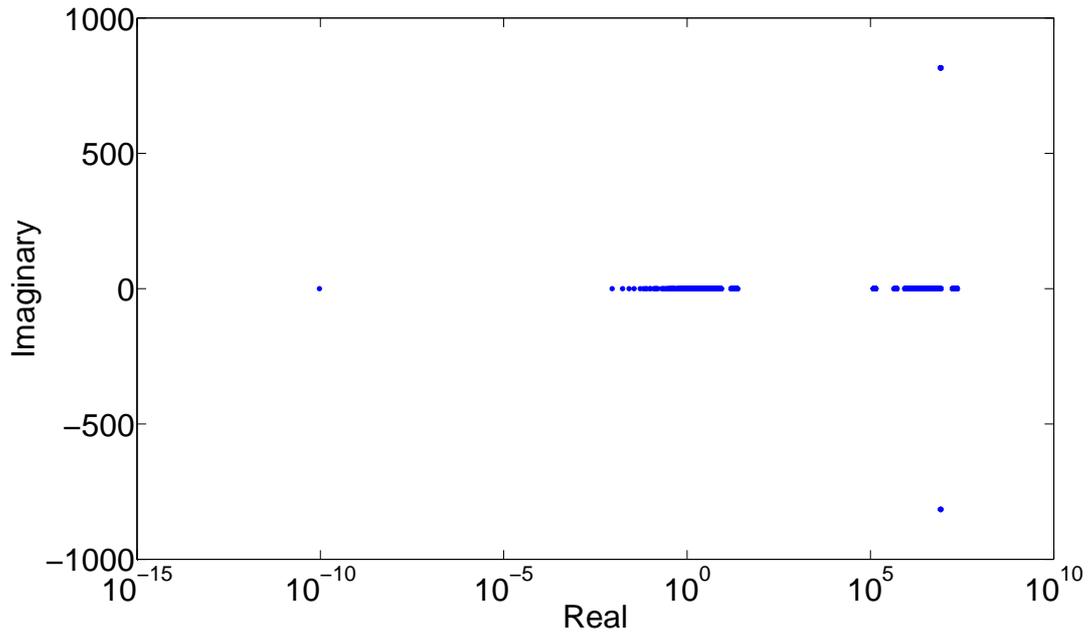}
%  \rule{4cm}{3cm}
  \caption{Eigenspectrum of the CRFVE stiffness matrix A for the distribution of $\alpha$ given in Example 4 with $\alpha_1=1e6$.}
 \label{fig:eigenvaluesofA}
 \end{figure}
 \begin{figure}
         \includegraphics[width=\linewidth]{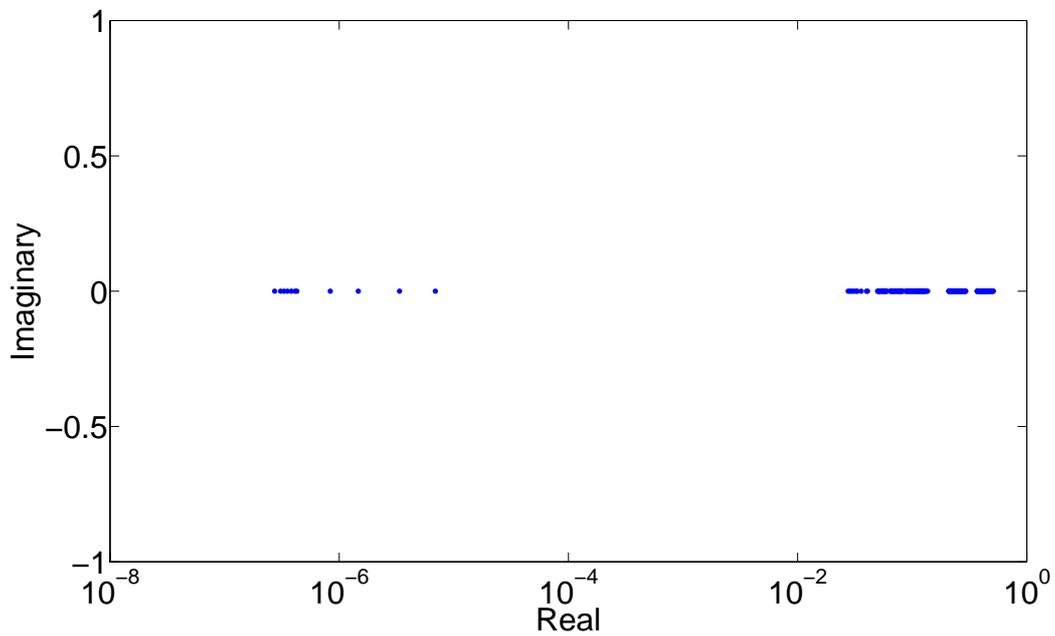}
%  \rule{4cm}{3cm}
  \caption{Eigenspectrum of the preconditioned operator, $T_A^{(2)}$, for the distribution of $\alpha$ given in Example 4 with $\alpha_1=1e6$.}
 \label{fig:eigenvaluesofprecondsystem}
% \caption{}
 \end{figure}
 
\begin{table}[htb]
\centering
\begin{tabular}{|l| l l l l l l |}
\hline
% &Average ASM&&&\\
% \hline
% $\alpha_1$&Example 1:&Example 2:&Example 3:&Example 4:\\
% \hline
$h/H$ &$1/4$&$1/8$&$1/16$&$1/32$&$1/64$&$1/128$\\
\hline
%   1     &17(11.41) \\
%   $10^0$&17(6.82)&21(12.20)& 1321\\
  $1/8$&22(1.89e-1)&&&&&\\
  $1/16$&32(8.80e-2)&25(1.67e-1)&&&&\\
  $1/32$&44(4.22e-2)&37(7.74e-2)&24(1.79e-1)&&&\\
  $1/64$&63(2.08e-2)&52(3.78e-2)&35(8.60e-2)&23(1.82e-1)&&\\
  $1/128$&89(1.03e-2)&74(1.87e-2)&49(4.21e-2)&33(8.95e-2)&21(1.83e-1)&\\
  $1/256$&126(5.12e-3)&106(9.30e-2)&69(2.09e-2)&46(4.42e-2)&29(9.05e-2)&18(1.83e-1)\\
\hline
\end{tabular}
\vspace{5mm}
\caption{Number of iterations  for the symmetric preconditioner for decreasing values of $h$ and $H$ with $\alpha=2+\sin(10\pi x)\sin(10\pi y)$. Estimates of the smallest eigenvalue of the symmetric part of the preconditioned system are reported in the parentheses.} 
\label{tbl:asstab1}
\end{table}
\begin{table}[htb]
\centering
\begin{tabular}{|l| l l l l l l |}
\hline
% &Average ASM&&&\\
% \hline
% $\alpha_1$&Example 1:&Example 2:&Example 3:&Example 4:\\
% \hline
$h/H$ &$1/4$&$1/8$&$1/16$&$1/32$&$1/64$&$1/128$\\
\hline
%   1     &17(11.41) \\
%   $10^0$&17(6.82)&21(12.20)& 1321\\
  $1/8$&20(1.90e-1)&&&&&\\
  $1/16$&30(9.24e-2)&24(1.79e-1)&&&&\\
  $1/32$&40(4.54e-2)&33(9.01e-2)&24(1.81e-1)&&&\\
  $1/64$&59(2.24e-2)&47(4.45e-2)&35(8.76e-2)&26(1.80e-1)&&\\
  $1/128$&83(1.11e-2)&68(2.19e-2)&49(4.37e-2)&39(8.76e-2)&28(1.70e-1)&\\
  $1/256$&116(5.50e-3)&95(1.09e-2)&68(2.16e-2)&55(4.29e-2)&41(8.21e-2)&27(1.78e-1)\\
%  $ 1/512$&*&*&*&*&*&40&26(4.53e-2)\\
\hline
\end{tabular}
\vspace{5mm}
\caption{Number of iterations  for the symmetric preconditioner for decreasing values of $h$ and $H$ with $\alpha=2+\sin(100\pi x)\sin(100\pi y)$. Estimates of the smallest eigenvalue of the symmetric part of the preconditioned system are reported in the parentheses.} 
\label{tbl:asstab2}
\end{table}
\begin{table}[htb]
\centering
\begin{tabular}{|l| l l l l l l |}
\hline
% &Average ASM&&&\\
% \hline
% $\alpha_1$&Example 1:&Example 2:&Example 3:&Example 4:\\
% \hline
$h/H$ &$1/4$&$1/8$&$1/16$&$1/32$&$1/64$&$1/128$\\
\hline
%   1     &17(11.41) \\
%   $10^0$&17(6.82)&21(12.20)& 1321\\
  $1/8$&19(1.91e-1)&&&&&\\
  $1/16$&27(9.23e-2)&22(1.83e-1)&&&&\\
  $1/32$&35(4.54e-2)&32(8.94e-2)&23(1.79e-1)&&&\\
  $1/64$&52(2.24e-2)&46(4.40e-2)&35(8.45e-2)&25(1.79e-1)&&\\
  $1/128$&75(1.11e-2)&62(2.20e-2)&46(4.38e-2)&37(8.73e-2)&28(1.70e-1)&\\
  $1/256$&107(5.50e-3)&89(1.10e-2)&64(2.16e-2)&53(4.28e-2)&40(8.23e-2)&26(1.77e-1)\\
%  $ 1/512$&*&*&*&*&*&40&26(4.53e-2)\\
\hline
\end{tabular}
\vspace{5mm}
\caption{Number of iterations  for the non-symmetric preconditioner for decreasing values of $h$ and $H$ with $\alpha=2+\sin(100\pi x)\sin(100\pi y)$. Estimates of the smallest eigenvalue of the symmetric part of the preconditioned system are reported in the parentheses.} 
\label{tbl:asstabNS}
\end{table}
The iteration numbers in Table \ref{tbl:itnmb} supports our theoretical results developed in Section \ref{sect:analysis}. We see no dependency on the contrast in $\alpha$ when the jumps are in the interior of subdomains, cf. Figure \ref{fig:alphaint}. If the coefficient has jumps in the subdomain layer, $\Omega_i^\delta$ corresponding to $\Omega_i$, the method is dependent on the ratio $\frac{\overline\alpha_i}{\underline{\alpha}_i}$, i.e., the ratio of the largest and smallest value of $\alpha$ in the layer, cf. Figure \ref{fig:alphabnd} and \ref{fig:alphachannelboundary}. When the jumps are only across the substructures, as in Figure \ref{fig:alphajumpsub}, the iteration numbers show that the method is robust with respect to the discontinuities in $\alpha$. 

The numerical results also show that the proposed method is asymptotically stable and scalable with respect to the dependence on $\frac{H_i}{h_i}$, and depends linearly on $\frac{H_i}{h_i}$ for the test cases under consideration as shown in Table~\ref{tbl:asstab1} and \ref{tbl:asstab2}. The coefficient $\alpha$ is here equal to $2+\sin(10\pi x)\sin(10\pi y)$ and $2+\sin(100\pi x)\sin(100\pi y)$, respectively. By comparing Table~\ref{tbl:asstab2} and \ref{tbl:asstabNS}, we see that the difference in behavior of the symmetric and non-symmetric preconditioner is negligible.

The distribution of the eigenvalues of the stiffness matrix $A$, as depicted in Figure~{\ref{fig:eigenvaluesofA}, include several complex eigenvalues with the magnitude of their complex part being close to zero, and two eigenvalues with multiplicity eight with a clearly visible complex part in the figure. The eigenvalues of the preconditioned operator, as depicted in Figure~\ref{fig:eigenvaluesofprecondsystem}, are all real and positive. Numerical testing have also shown that for the test cases where our theory predicts dependency on the coefficient jump in $\alpha$, the smallest eigenvalues of the symmetric part of the preconditioned operator, $T_A^{(2)}$, are inversely proportional to the ratio $\frac{\overline\alpha_i}{\underline{\alpha}_i}$.

In Figure \ref{fig:relresplot1}--\ref{fig:relresplot4} we have plotted the relative residuals and the relative preconditioned residuals measured in the $l_2$ norm. These plots show that if the stopping criteria is based on the preconditioned residual the method will in the worst case converge to the prescribed tolerance even though the resulting GMRES solution of the linear system is far from the exact solution. Hence, using a stopping criteria based on the $l_2$ norm of the residual instead of the more commonly used $l_2$ norm of the preconditioned residual is in our case a much more viable choice. 

Finally, we conclude this section by stating that the numerical results presented here confirm the theory developed in the previous sections regarding the non-symmetry of the finite volume element stiffness matrix, the estimate for the convergence rate of the GMRES method applied to our preconditioned systems (\ref{eq:prcndsystem}) for $k=2,3$ and the convergence estimate for the difference between the FE and the FVE bilinear form, cf. Equation (\ref{eq:difffefve}).
\begin{figure}[htb]
\centering
\begin{subfigure}[t]{0.48\linewidth}
        \centering
         \includegraphics[trim=46 10 115 20, clip,width=\linewidth]{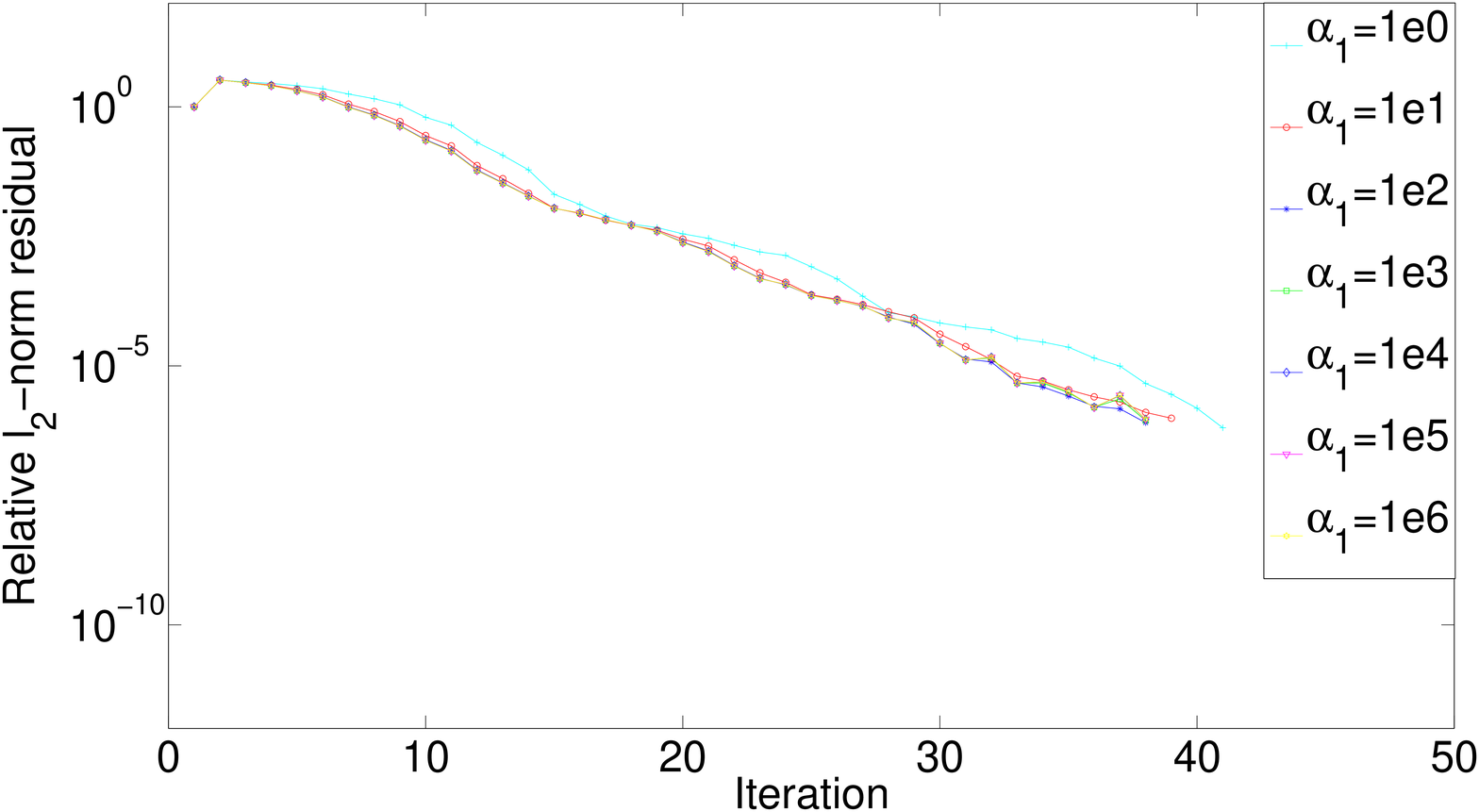}
%  \rule{4cm}{3cm}
  \caption{Example 1. Relative residual norms for GMRES minimizing the $a$-norm for different $\alpha_1$.}
 \label{fig:resplotchannelinclusioninside1}

\end{subfigure}
\hspace{0.1cm}
\begin{subfigure}[t]{0.48\linewidth}
        \centering
%  \rule{4cm}{3cm}
 \includegraphics[trim=46 10 115 20, clip,width=\linewidth]{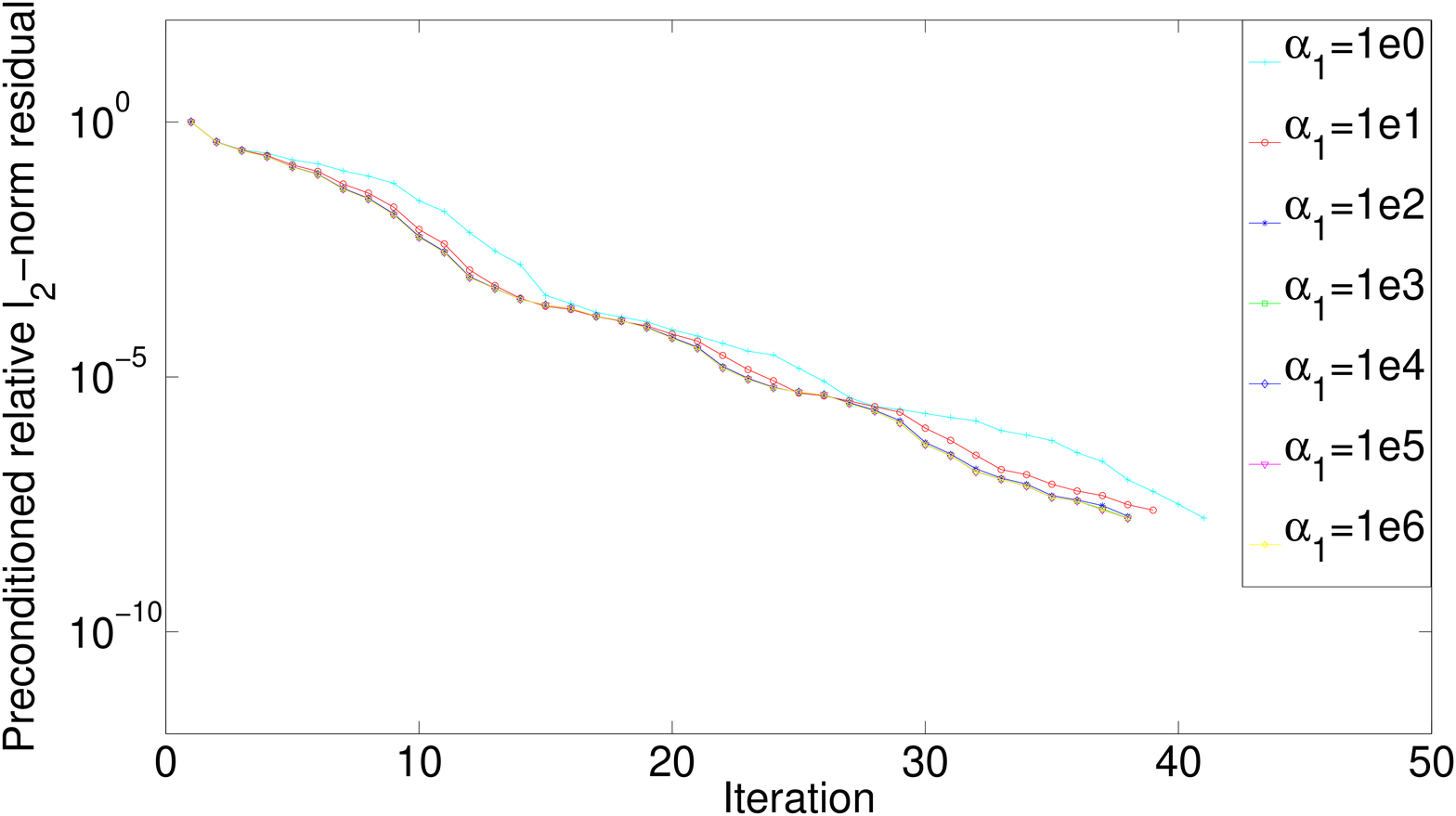}
 \caption{Example 1. Relative preconditioned residual norms for GMRES minimizing the $a$-norm for different $\alpha_1$.}
 \label{fig:resplotchannelinclusioninside2}
\end{subfigure}
% \caption{Example 1.}
% \end{figure}
\noindent
% \begin{figure}
\centering
\caption{}
\label{fig:relresplot1}
\end{figure}

\begin{figure}[htb]
\begin{subfigure}[t]{0.48\linewidth}
        \centering
         \includegraphics[trim=46 10 115 20, clip,width=\linewidth]{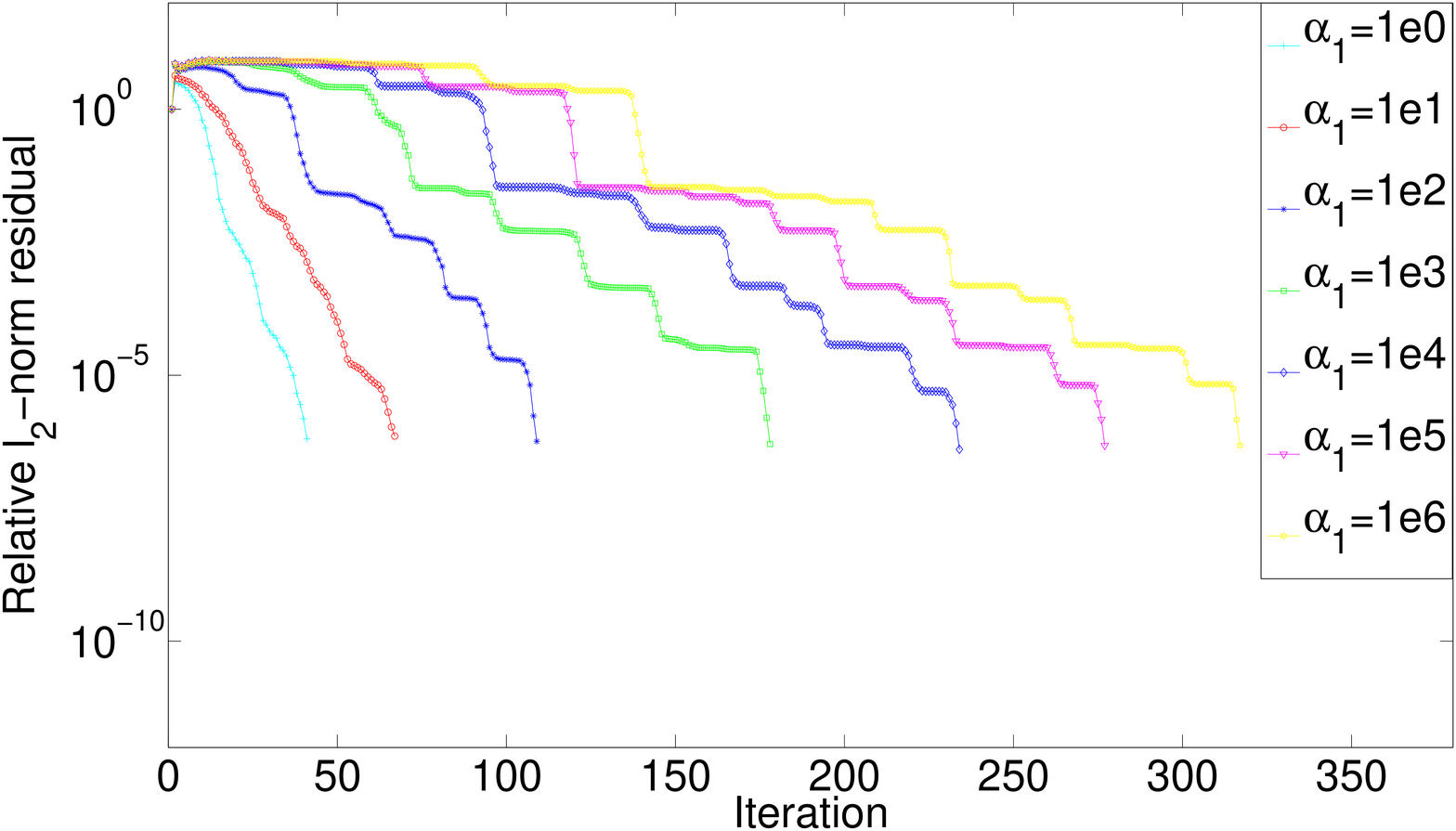}
%  \rule{4cm}{3cm}
  \caption{Example 2. Relative residual norms for GMRES minimizing the $a$-norm for different $\alpha_1$.}
 \label{fig:resplotalphabnd1}

\end{subfigure}
\hspace{0.1cm}
\begin{subfigure}[t]{0.48\linewidth}
        \centering
%  \rule{4cm}{3cm}
 \includegraphics[trim=46 10 115 20, clip,width=\linewidth]{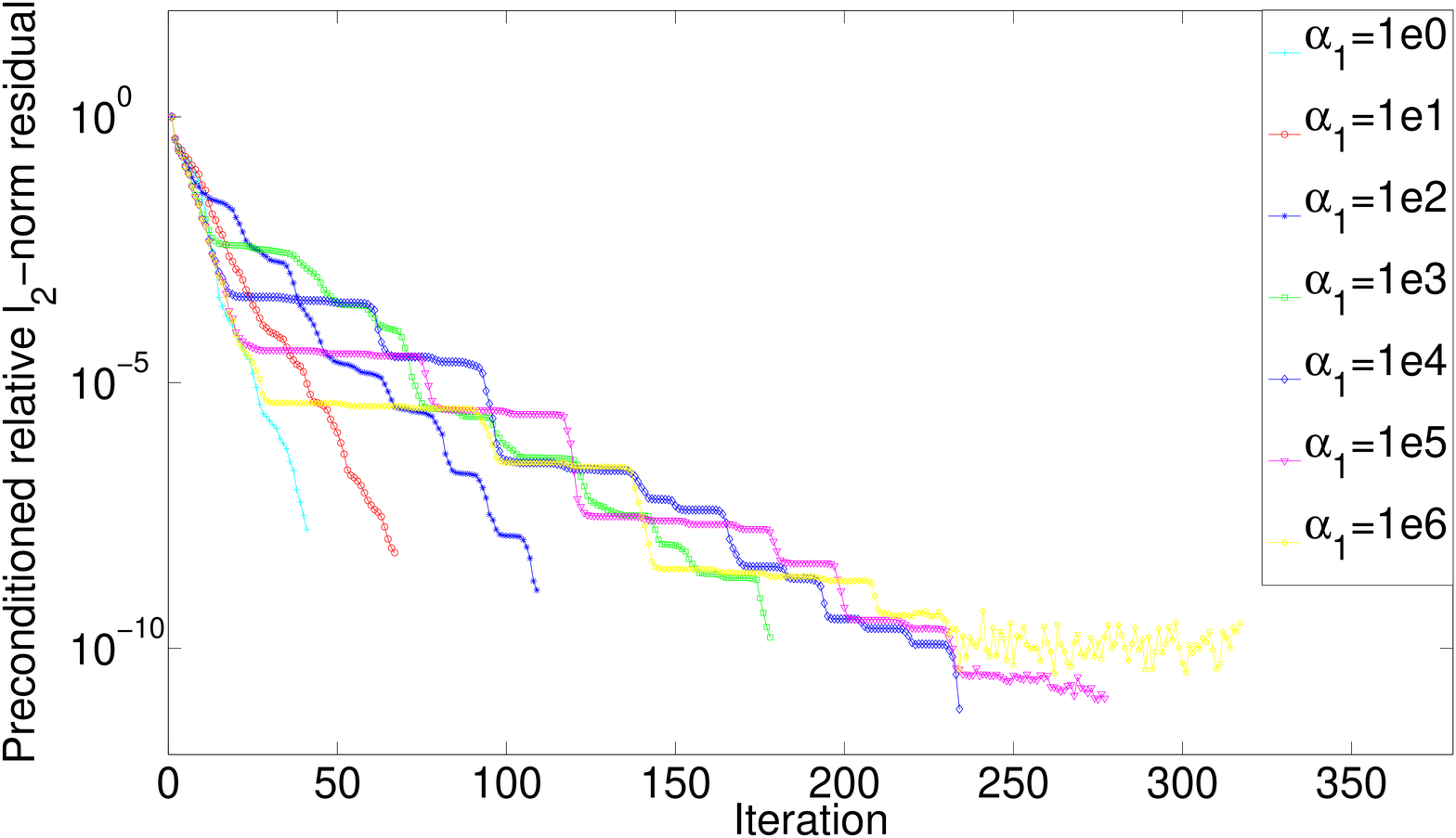}
 \caption{Example 2. Relative preconditioned residual norms for GMRES minimizing the $a$-norm for different $\alpha_1$.}
 \label{fig:resplotalphabnd2}

\end{subfigure}
% \caption{Example 2.}
% \end{figure}
\caption{}
\label{fig:relresplot2}
\end{figure}

\begin{figure}[htb]
\noindent

% \begin{figure}
\centering
\begin{subfigure}[t]{0.48\linewidth}
        \centering
         \includegraphics[trim=46 10 115 20, clip,width=\linewidth]{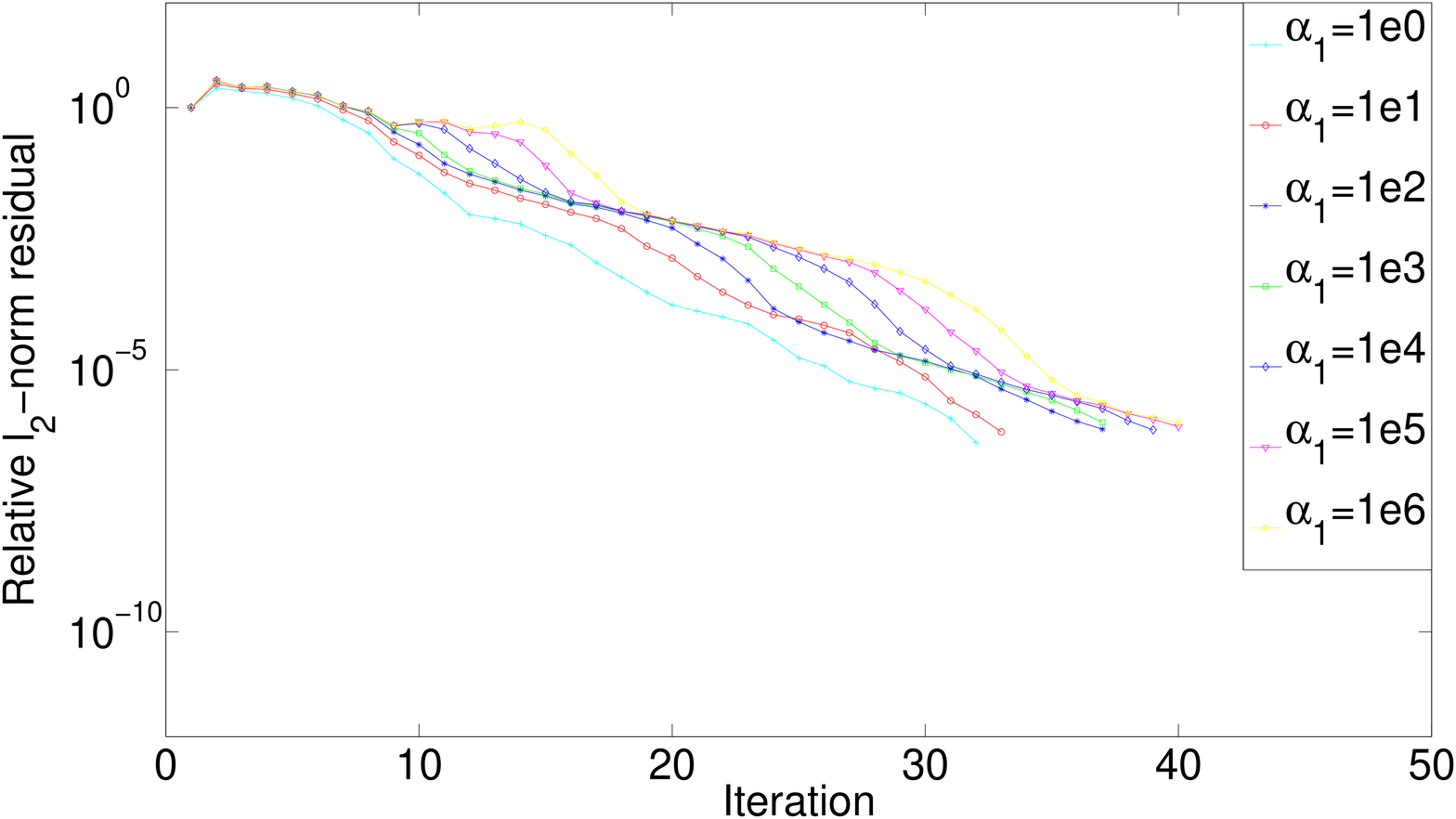}
%  \rule{4cm}{3cm}
  \caption{Example 3. Relative residual norms for GMRES minimizing the $a$-norm for different $\alpha_1$.}
 \label{fig:resplotalphabnd3}

\end{subfigure}
\hspace{0.1cm}
\begin{subfigure}[t]{0.48\linewidth}
        \centering
%  \rule{4cm}{3cm}
 \includegraphics[trim=46 10 115 20, clip,width=\linewidth]{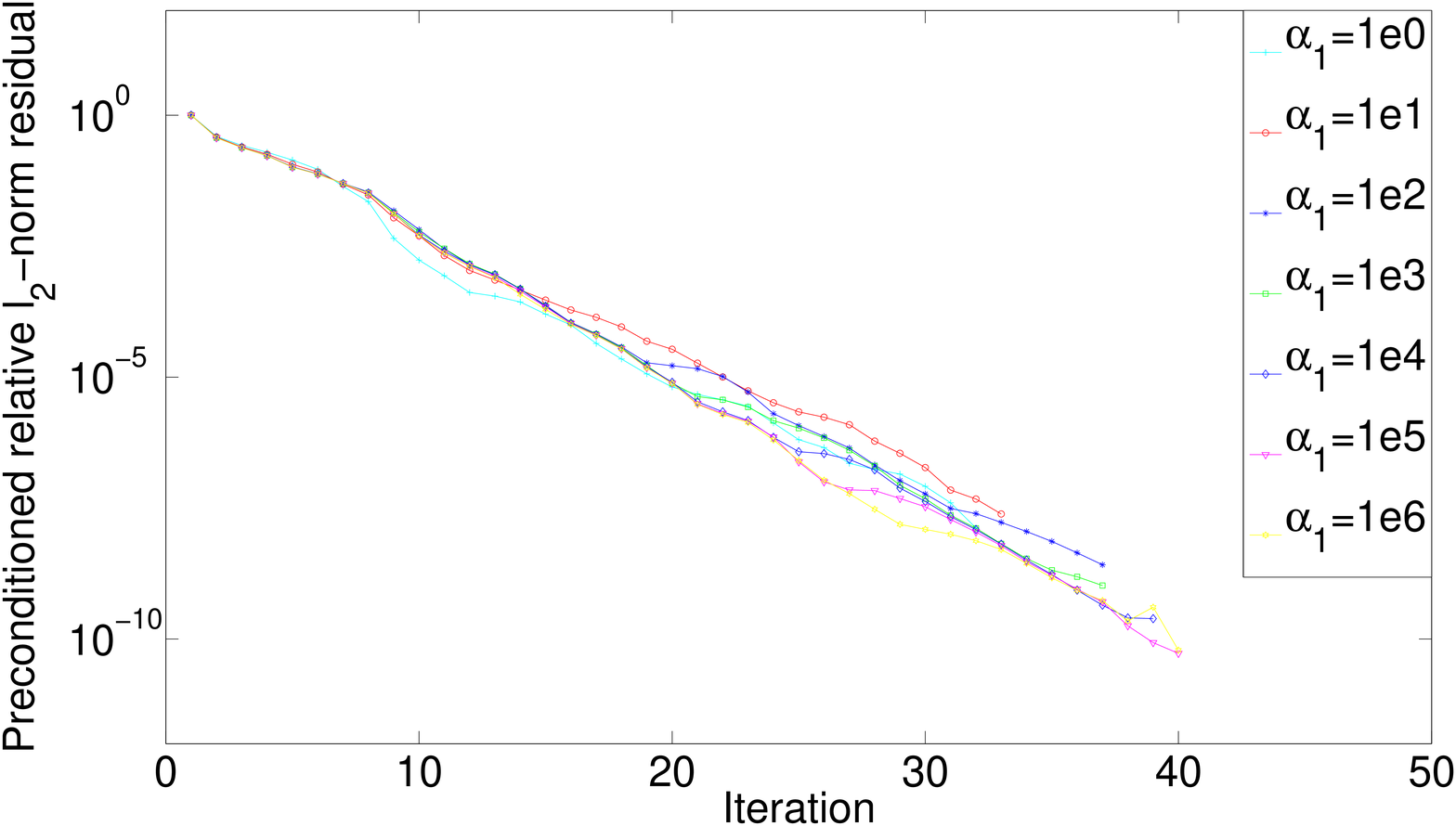}
 \caption{Example 3. Relative preconditioned residual norms for GMRES minimizing the $a$-norm for different $\alpha_1$.}
 \label{fig:resplotalphabnd4}
\end{subfigure}
\caption{}
\label{fig:relresplot3}
\end{figure}

\begin{figure}[htb]

%  \caption{Plot of the $l_2$ norm of the relative residual (right) and the $l_2$ of the preconditioned relative residual (left) for GMRES minimizing the $A$-norm.}
\begin{subfigure}[t]{0.48\linewidth}
        \centering
         \includegraphics[trim=46 10 115 20, clip,width=\linewidth]{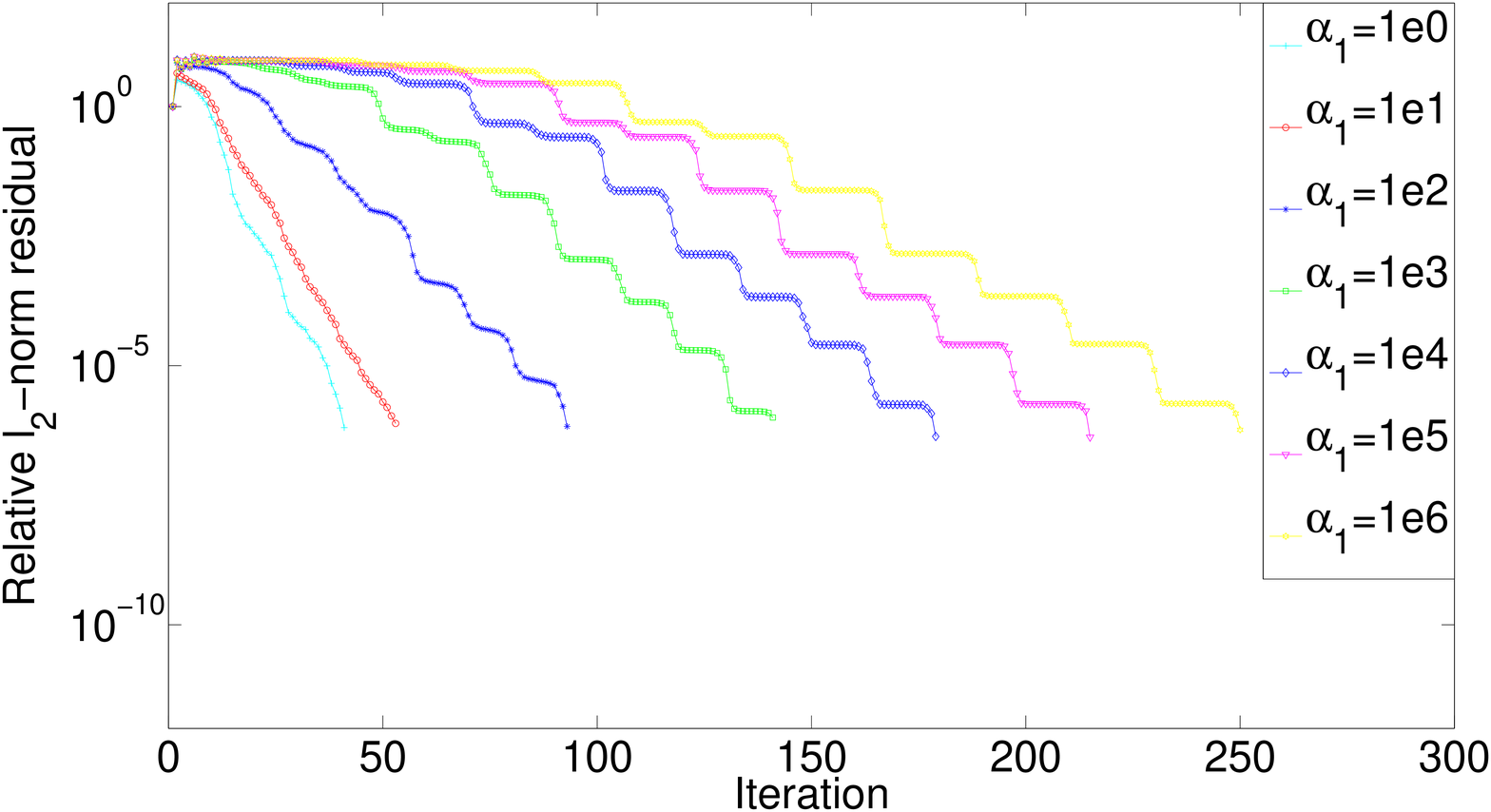}
%  \rule{4cm}{3cm}
  \caption{Example 4. Relative residual norms for GMRES minimizing the $a$-norm for different $\alpha_1$.}
 \label{fig:resplotalphachannelboundary1}

\end{subfigure}
\hspace{0.1cm}
\begin{subfigure}[t]{0.48\linewidth}
        \centering
%  \rule{4cm}{3cm}
 \includegraphics[trim=46 10 115 20, clip,width=\linewidth]{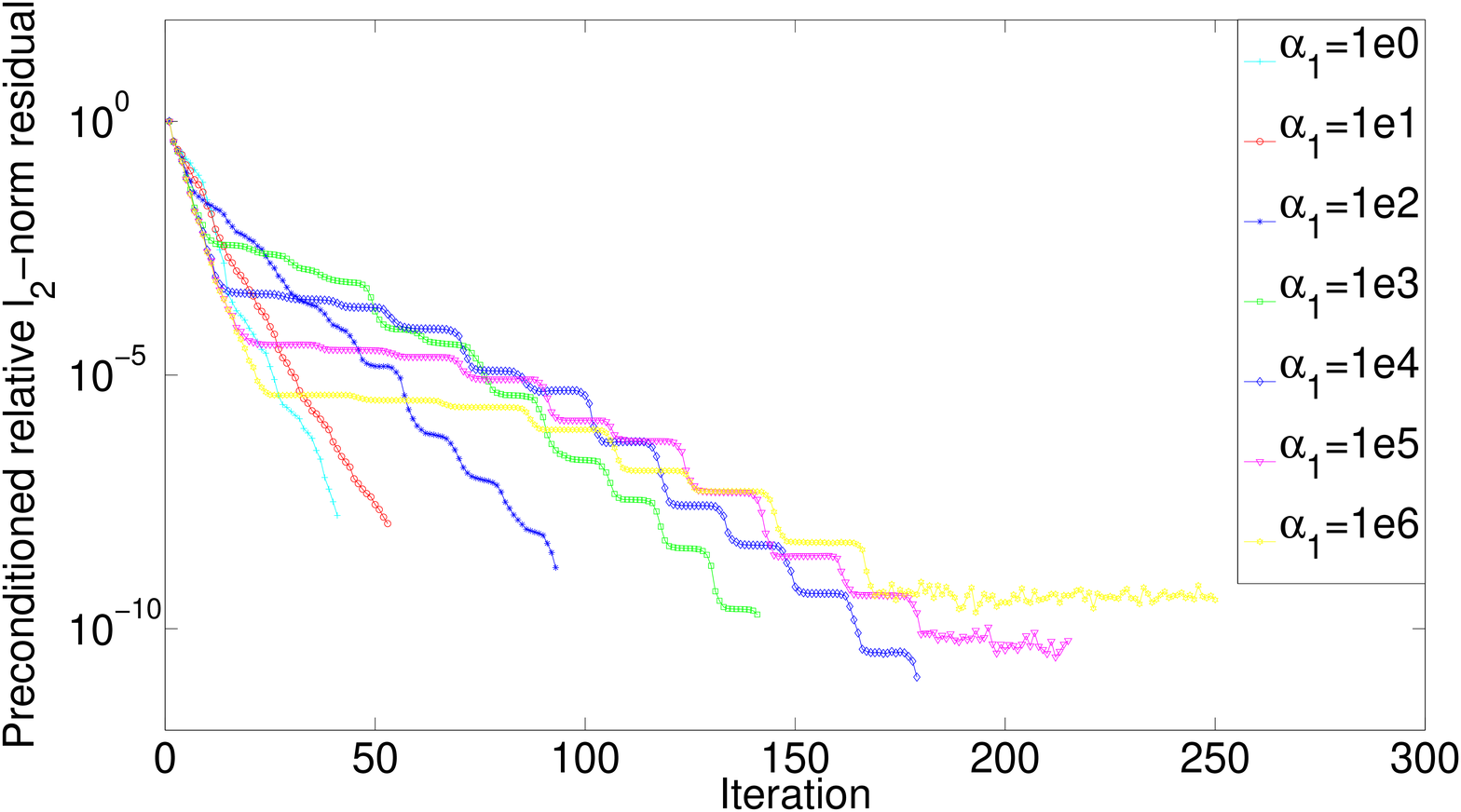}
 \caption{Example 4. Relative preconditioned residual norms for GMRES minimizing the $a$-norm for different $\alpha_1$.}
 \label{fig:resplotalphachannelboundary2}
\end{subfigure}
\caption{}
\label{fig:relresplot4}

\end{figure}
\clearpage

 \bibliography{CRFVEAddAvg-LMR}

\begin{thebibliography}{10}

\bibitem{BANK:1987:BOX}
Randolph~E. Bank and Donald~J. Rose.
\newblock Some error estimates for the box method.
\newblock {\em SIAM J. Numer. Anal.}, 24(4):777--787, 1987.

\bibitem{BJORSTAD:1996:AVG}
Petter~E. Bj{\o}rstad, Maksymilian Dryja, and Eero Vainikko.
\newblock Additive {S}chwarz methods without subdomain overlap and with new
  coarse spaces.
\newblock In {\em Domain decomposition methods in sciences and engineering
  ({B}eijing, 1995)}, pages 141--157. Wiley, Chichester, 1997.

\bibitem{Brenner:1996:TLS}
Susanne~C. Brenner.
\newblock Two-level additive {S}chwarz preconditioners for nonconforming finite
  element methods.
\newblock {\em Math. Comp.}, 65(215):897--921, 1996.

\bibitem{brenner2008mathematical}
Susanne~C Brenner and Larkin~Ridgway Scott.
\newblock {\em The mathematical theory of finite element methods}, volume~15.
\newblock Springer, 2008.

\bibitem{Brenner:1999:BDD}
Susanne~C. Brenner and Li-Yeng Sung.
\newblock Balancing domain decomposition for nonconforming plate elements.
\newblock {\em Numer. Math.}, 83(1):25--52, 1999.

\bibitem{cai1989some}
Xiao-Chuan Cai.
\newblock {\em Some domain decomposition algorithms for nonselfadjoint elliptic
  and parabolic partial differential equations}.
\newblock ProQuest LLC, Ann Arbor, MI, 1989.
\newblock Thesis (Ph.D.)--New York University.

\bibitem{cai1992domain}
Xiao-Chuan Cai and Olof~B. Widlund.
\newblock Domain decomposition algorithms for indefinite elliptic problems.
\newblock {\em SIAM J. Sci. Statist. Comput.}, 13(1):243--258, 1992.

\bibitem{CAI:1991:ON}
Zhi~Qiang Cai.
\newblock On the finite volume element method.
\newblock {\em Numer. Math.}, 58(7):713--735, 1991.

\bibitem{chatzipantelidis1999finite}
Panagiotis Chatzipantelidis.
\newblock A finite volume method based on the crouzeix--raviart element for
  elliptic pde's in two dimensions.
\newblock {\em Numerische Mathematik}, 82(3):409--432, 1999.

\bibitem{Chatzipantelidis:2002:FVE}
Panagiotis Chatzipantelidis.
\newblock Finite volume methods for elliptic {PDE}'s: a new approach.
\newblock {\em M2AN Math. Model. Numer. Anal.}, 36(2):307--324, 2002.

\bibitem{chou2003domain}
SH~Chou and J~Huang.
\newblock A domain decomposition algorithm for general covolume methods for
  elliptic problems.
\newblock {\em Journal of Numerical Mathematics jnma}, 11(3):179--194, 2003.

\bibitem{Dolean:2012:DirNeu}
Victorita Dolean, Fr{\'e}d{\'e}ric Nataf, Robert Scheichl, and Nicole Spillane.
\newblock Analysis of a two-level {S}chwarz method with coarse spaces based on
  local {D}irichlet-to-{N}eumann maps.
\newblock {\em Comput. Methods Appl. Math.}, 12(4):391--414, 2012.

\bibitem{dryja2010additive}
Maksymilian Dryja and Marcus Sarkis.
\newblock Additive average schwarz methods for discretization of elliptic
  problems with highly discontinuous coefficients.
\newblock {\em Comput. Methods Appl. Math.}, 10(2):164--176, 2010.

\bibitem{eisenstat1983variational}
Stanley~C Eisenstat, Howard~C Elman, and Martin~H Schultz.
\newblock Variational iterative methods for nonsymmetric systems of linear
  equations.
\newblock {\em SIAM Journal on Numerical Analysis}, 20(2):345--357, 1983.

\bibitem{ewing2002accuracy}
R.E. Ewing, T.~Lin, and Y.~Lin.
\newblock On the accuracy of the finite volume element method based on
  piecewise linear polynomials.
\newblock {\em SIAM Journal on Numerical Analysis}, pages 1865--1888, 2002.

\bibitem{Galvis:2010:DDMULT1}
Juan Galvis and Yalchin Efendiev.
\newblock Domain decomposition preconditioners for multiscale flows in
  high-contrast media.
\newblock {\em Multiscale Model. Simul.}, 8(4):1461--1483, 2010.

\bibitem{Galvis:2010:DDMULT2}
Juan Galvis and Yalchin Efendiev.
\newblock Domain decomposition preconditioners for multiscale flows in high
  contrast media: reduced dimension coarse spaces.
\newblock {\em Multiscale Model. Simul.}, 8(5):1621--1644, 2010.

\bibitem{Graham:2007:MULTDD}
I.~G. Graham, P.~O. Lechner, and R.~Scheichl.
\newblock Domain decomposition for multiscale {PDE}s.
\newblock {\em Numer. Math.}, 106(4):589--626, 2007.

\bibitem{HACKBUSCH:1989:FIRST}
Wolfgang Hackbusch.
\newblock On first and second order box schemes.
\newblock {\em Computing}, 41(4):277--296, 1989.

\bibitem{Hou:1997:MultFEM}
Thomas~Y. Hou and Xiao-Hui Wu.
\newblock A multiscale finite element method for elliptic problems in composite
  materials and porous media.
\newblock {\em J. Comput. Phys.}, 134(1):169--189, 1997.

\bibitem{LIN:2013:FVEM}
Yanping Lin, Jiangguo Liu, and Min Yang.
\newblock Finite volume element methods: an overview on recent developments.
\newblock {\em Int. J. Numer. Anal. Model. Ser. B}, 4(1):14--34, 2013.

\bibitem{marcinkowski1999mortar}
Leszek Marcinkowski.
\newblock The mortar element method with locally nonconforming elements.
\newblock {\em BIT Numerical Mathematics}, 39(4):716--739, 1999.

\bibitem{Marcinkowski:2014:ASMFVE}
Leszek Marcinkowski, Talal Rahman, and Jan Valdman.
\newblock Additive schwarz preconditioner for the general finite volume element
  discretization of symmetric elliptic problems.
\newblock May 2014.
\newblock Published online in arXiv:1405.0185 [math.NA].

\bibitem{rahman2005additive}
Talal Rahman, Xuejun Xu, and Ronald Hoppe.
\newblock Additive schwarz methods for the crouzeix-raviart mortar finite
  element for elliptic problems with discontinuous coefficients.
\newblock {\em Numerische Mathematik}, 101(3):551--572, 2005.

\bibitem{rui2008convergence}
Hongxing Rui and Chunjia Bi.
\newblock Convergence analysis of an upwind finite volume element method with
  crouzeix-raviart element for non-selfadjoint and indefinite problems.
\newblock {\em Frontiers of Mathematics in China}, 3(4):563--579, 2008.

\bibitem{saad1986gmres}
Yousef Saad and Martin~H Schultz.
\newblock Gmres: A generalized minimal residual algorithm for solving
  nonsymmetric linear systems.
\newblock {\em SIAM Journal on scientific and statistical computing},
  7(3):856--869, 1986.

\bibitem{sarkis1997nonstandard}
Marcus Sarkis.
\newblock Nonstandard coarse spaces and schwarz methods for elliptic problems
  with discontinuous coefficients using non-conforming elements.
\newblock {\em Numerische Mathematik}, 77(3):383--406, 1997.

\bibitem{smith1996domain}
Barry Smith, Petter Bjorstad, and William Gropp.
\newblock {\em Domain decomposition: parallel multilevel methods for elliptic
  partial differential equations}.
\newblock Cambridge University Press, 1996.

\bibitem{Spillane:2014:GENEO}
N.~Spillane, V.~Dolean, P.~Hauret, F.~Nataf, C.~Pechstein, and R.~Scheichl.
\newblock Abstract robust coarse spaces for systems of {PDE}s via generalized
  eigenproblems in the overlaps.
\newblock {\em Numer. Math.}, 126(4):741--770, 2014.

\bibitem{toselli2005domain}
Andrea Toselli and Olof~B Widlund.
\newblock {\em Domain decomposition methods: algorithms and theory}, volume~34.
\newblock Springer, 2005.

\bibitem{WU:2003:ERROR}
Haijun Wu and Ronghua Li.
\newblock Error estimates for finite volume element methods for general
  second-order elliptic problems.
\newblock {\em Numer. Methods Partial Differential Equations}, 19(6):693--708,
  2003.

\bibitem{Zhang:2006:ODD}
Sheng Zhang.
\newblock On domain decomposition algorithms for covolume methods for elliptic
  problems.
\newblock {\em Comput. Methods Appl. Mech. Engrg.}, 196(1-3):24--32, 2006.

\end{thebibliography}
% \addcontentsline{toc}{chapter}{Bibliography}
 \bibliographystyle{plain}
\end{document}